\documentclass[11pt,reqno]{amsart}

 \usepackage{amsmath,amsthm,amsfonts}
\usepackage{latexsym,mathabx,txfonts}
\usepackage{amssymb}
\usepackage{slashed}
\usepackage{tikz}

\newcommand{\R}{\mathbb{R}}

\newcommand{\N}{\mathbb{N}}
\newcommand{\Z}{\mathbb{Z}}

\renewcommand{\hat}{\widehat}

\numberwithin{equation}{section}

\newtheorem{thm}{Theorem}[section]

\newtheorem{lem}[thm]{Lemma}
\newtheorem{prop}[thm]{Proposition}

\newtheorem{assum}{Assumption}
\theoremstyle{remark}
\newtheorem{rem}{Remark}[section]

\newcommand{\Del}[1]{}

\begin{document}

\title[Randomization improved Strichartz estimates]{Randomization improved Strichartz estimates and global well-posedness for supercritical data}
\author{Nicolas Burq and Joachim Krieger}



\begin{abstract}
We introduce a novel data randomisation for the free wave equation which leads to the same range of Strichartz estimates as for radial data, albeit in a non-radial context. We then use these estimates to establish global well-posedness for a wave maps type nonlinear wave equation for certain supercritical data, provided the data are suitably small and randomised. 
\end{abstract}

\maketitle

\section{Improving Strichartz estimates via suitable randomization}

Consider the free wave equation
\[
\Box u = -u_{tt} + \triangle u = 0
\]
on $\R^{n+1}$, where we shall restrict to the case $n\geq 2$. Denote the initial data by $u[0] = (u(0, \cdot), u_t(0, \cdot))$. Interpolation of the point wise decay and energy conservation lead to the famous Strichartz estimates 
\begin{equation}\label{eq:Str}
\big\|u\big\|_{L_{t}^p L_x^q}\leq C(p, q, n)\big\|u[0]\big\|_{\dot{H}^{\gamma}(\R^n)\times \dot{H}^{\gamma - 1}(\R^n)},\,\frac{1}{p} + \frac{n}{q} = \frac{n}{2} - \gamma, 
\end{equation}
provided we restrict $(p, q)$ to the Strichartz admissible range, given by $\frac{1}{p} + \frac{n-1}{2q}\leq \frac{n-1}{4}$, $p\geq 2$, with the case $(n, p, q) = (3, 2, \infty)$ excluded. These estimates have been known to be optimal in general due to the well-known Knapp counterexamples. 
\\
However, it has also been known for a while that the latter can be avoided by imposing either a symmetry reduction, such as radiality (\cite{KlMa}), or imposing additional constraints on the angular regularity of the data (\cite{Ste}), in which case the range of available Strichartz estimates can be significantly improved to 
\begin{equation}\label{eq:radialStr}
\frac{1}{p} + \frac{n-1}{q}<\frac{n-1}{2},\,p\geq 2.
\end{equation}

\begin{figure}
\begin{tikzpicture}[scale=1.5]
\draw [blue,fill=blue!60](-2,2)--(-.2,0)--(1.6,0)--cycle;
\draw [blue,fill=blue!20](-2,0)--(-2,2)--(-.2,0)--cycle;
\draw[thick,  ->] (-2.2,0) -- (2.2,0);
\draw[thick,  ->] (-2, -.2) -- (-2,2.5);
\draw (2, 0) node[below] {$\frac 1 p$};
\draw (1.6, 0) node[below] {$\frac {n-1} 2$};
\draw (-.2, 0) node[below] {$\frac {n-1} 4 $};
\draw (-2, 2.4) node[left] {$\frac 1 q$};
\draw (-2, 2) node[left] {$\frac 1 2$};
\end{tikzpicture}
\caption{Admissible ranges: light = general range, dark = extended range} 
\end{figure}
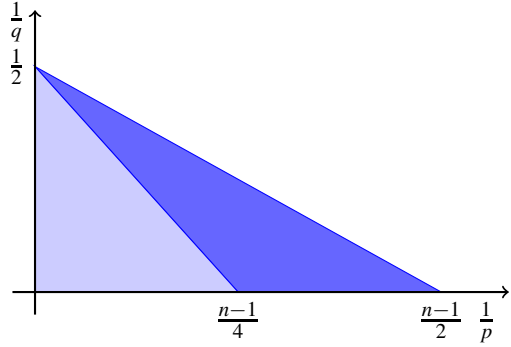
This section contains the observation that combining the method of proof from \cite{Ste} with the asymptotic analysis of 'generic' orthonormal bases for the space of spherical harmonics on $S^{n-1}$ in \cite{BurqLeb} and implementing a suitable randomisation, one can obtain almost the same estimates as in the radial case, see Proposition~\ref{prop:randomStrich1} below. In the following section, we shall show how one can use a refinement of these estimates (Proposition~\ref{prop:randomStrich2}) to deduce small data global well-posedness results below the critical scaling for certain nonlinear wave equations of 'fractional derivative Wave Maps type' on $\R^{3+1}$. For a recent work on well-posedness of derivative nonlinear wave equations involving randomised data see \cite{Czubak}. 

\subsection{Probabilistic orthonormal frames for spherical harmonics}

Consider the sphere $S^d\hookrightarrow \R^{d+1}$ and let $\triangle $ be the Laplace operator with respect to its canonical metric. Eigenfunctions $u$ satisfying $-\triangle u = \lambda^2 u$ satisfy the well-known estimates 
\[
\|u\|_{L^p(M)}\leq C\lambda^{\delta(p)}\|u\|_{L^2(M)},\,2\leq p\leq\infty,
\]
where 
\[
\delta(p) = \frac{d-1}{2} - \frac{d}{p},\,p\geq \frac{2(d+1)}{2-1)},\,\delta(p) = \frac{d-1}{2}(\frac12 - \frac{1}{p}),\,p\leq \frac{2(d+1)}{(d-1)},
\]
and these bounds are in fact optimal on the sphere. However, from \cite{BurqLeb}, we infer that eigenfunctions saturating the preceding bounds are in some sense exceptional, and that in fact orthonormal bases for $L^2(S^d)$ may be constructed which much improve these bounds. 

\begin{thm}\label{thm:almost bounded}(Burq-Lebeau \cite{BurqLeb-1, BurqLeb}) Denote by $E_k$ the space of spherical harmonics of dimension $N_k$ and associated to eigenvalue $\lambda^2 = k(k+d-1)$, $k\in \N$. Identify the set of orthonormal frames $B_k = (b_{k,l})_{l=1}^{N_k}$ of $E_k$ with the orthogonal group $O(N_k)$, equipped with the Haar measure $\nu_k$.  Let $\nu = \otimes _k \nu_k$ the natural probability measure on the set of sequence of orthonormal frames $B= (B_k)$. Then there are constants $C, c, c_0>0$ and for all $q<+\infty$, constants 
$$ \mathcal{M}_{q,k} \sim C \sqrt{q}  \text{ when $k \rightarrow + \infty$} $$ 
such that 
\begin{equation}
\begin{aligned}
\nu \big(\{B = (b_{k,l}); \exists k, l; \|b_{k,l}\|_{L^\infty(S^d)}>(c_0 + r)\sqrt{\log k}\}\big)& \leq C e^{-cr^2}\\
   \nu \big(\{B = (b_{k,l}); \exists k, l; \|b_{k,l}\|_{L^q(S^d)} - \mathcal{M}_{q,k} > r \}\big)& \leq C e^{-cr^2}
 \end{aligned}
 \end{equation}In particular, we can select an orthonormal frame $\{b_{k,l}; l = 1,2,\ldots, N_k,\,k = 1,2,\ldots,\}$  for $L^2(S^d)$ consisting of eigenfunctions of $\triangle$ with the property that
\begin{equation}\label{eq:bbound}
 \exists C; \forall k, \forall l =0, \dots, N_k, \quad \|b_{k,l}\|_{L^q(S^d)}\leq \begin{cases}C\sqrt{\log k}, & \text{ if } q= +\infty\\
 C\sqrt{q}, & \text{ if } q< +\infty. 
\end{cases}
\end{equation}
\end{thm}

Call such a frame $\{b_{k,l};k,l\geq 1\}$ a {\it{good frame}}. We shall use such a good frame to implement a suitable data randomisation in the sequel. 

\subsection{Randomization improved Strichartz estimates}
\subsubsection{Using {{good frames}}}\label{sec.1.2.1}
Pick a good frame $\{b_{k,l}\}_{k,l\geq 1}$ for $L^2(S^{n-1})$. Consider a function $f(x)$ on $\R^n$ supported at frequency $\sim 1$, and write its Fourier transform in terms of the good frame after passage to spherical coordinates $\rho, \theta$: 
\[
\hat{f}(\rho\theta) = \sum_{k,l}\hat{c}_{k,l}(\rho)b_{k,l}(\theta)
\]
In turn, this gives a representation of $f(x)$ in terms of the good basis as follows (see~\cite[Theorem 3.10]{StWe}):
\begin{align}\label{rep}
f(r\theta) = \sum_{k,l}2\pi i^kr^{\frac{2-n}{2}}b_{k,l}(\theta)\cdot\int_0^\infty J_{\frac{n-2}{2}+k}(2\pi r\rho)\hat{c}_{k,l}(\rho)\rho^{\frac{n}{2}}d\rho,
\end{align}
and we have 
\begin{equation}\label{rep2}
\sum_{k,l}\big\|\hat{c}_{k,l}(\rho)\big\|^2_{L^2_{d\rho}}\sim \big\|f\big\|_{L^2(\R^n)}^2. 
\end{equation}

Now let $h_{k,l}(\tilde{\omega})$ be a collection of real-valued independent random variables with distributions $\mu_{k,l}$ on some probability space, satisfying for some $c>0$ and all $\gamma\in \R$ the bounds 
\[
\big|\int_{\R}e^{\gamma x}d\mu_{k,l}\big|\leq e^{c\gamma^2},\,k,l\geq 1. 
\]
Introduce the functions 
\[
\hat{f}^{(\tilde{\omega})}(\rho\omega) := \sum_{k,l}h_{k,l}(\tilde{\omega})\hat{c}_{k,l}(\rho)b_{k,l}(\theta)
\]
Also, denote their inverse Fourier transform by $f^{(\tilde{\omega})}(r\omega)$. Finally, let 
\[
u^{(\tilde{\omega})}(t, x) = \big(e^{-it\sqrt{-\triangle}}f^{(\tilde{\omega})}\big)(x). 
\]
Then we can state the following
\begin{prop}\label{prop:randomStrich1} Let $(2, q)$ be admissible in the sense of \eqref{eq:radialStr}. Then for suitable positive constants $c, C$ we have 
\begin{equation}
\mathbf{P}\big(\{\big\|u^{(\tilde{\omega})}\big\|_{L_t^2 L_x^q(\R^{n+1})}>\lambda\}\big)\leq \begin{cases} C e^{-c\frac{\lambda^2}{\| f\|_{L^2(\R^n)}^2}} &\text{ if } q<+\infty\\
C e^{-c\frac{\lambda^2}{\|f\|_{H^{0+}(\R^n)}^2}} &\text{ if } q=+\infty
\end{cases}
\end{equation}
\end{prop}
\begin{proof} If $(2,+\infty)$ is admissible in the sense of \eqref{eq:radialStr}, then $(2,q)$ is also admissible for $q$ sufficiently large and the second estimate follows from Sobolev embeddings $W^{\epsilon, q} \rightarrow L^\infty$ for $\epsilon >0$ and $q$ sufficiently large. We now assume $q<\infty$.  Write (with $\omega\in S^{n-1}$)
\begin{align*}
u^{(\tilde{\omega})}(t, r\omega) = \sum_{k,l}2\pi i^kr^{\frac{2-n}{2}}h_{k,l}(\tilde{\omega})b_{k,l}(\theta)\cdot\int_0^\infty e^{-2\pi i t\rho}J_{\frac{n-2}{2}+k}(2\pi r\rho)\hat{c}_{k,l}(\rho)\rho^{\frac{n}{2}}d\rho
\end{align*}
Then using Minkowski's inequality and Lemma 3.1 from \cite{BurqTzve}, we get for~$s\geq q$ 
\begin{equation}\label{eq:1}\begin{split}
&\big\|u^{(\tilde{\omega})}\big\|_{L^s_{\tilde{\omega}}L_t^2 L_x^q}\\&\leq C\sqrt{s}\big\|\big(\sum_{k,l}\big|r^{\frac{2-n}{2}}b_{k,l}(\theta)\cdot\int_0^\infty e^{-2\pi i t\rho}J_{\frac{n-2}{2}+k}(2\pi r\rho)\hat{c}_{k,l}(\rho)\rho^{\frac{n}{2}}d\rho\big|^2\big)^{\frac12}\big\|_{L_t^2 L_x^q}
\\
&\leq C\sqrt{s}\Big(\sum_{k,l}\big\| r^{\frac{2-n}{2}}b_{k,l}(\theta)\cdot\int_0^\infty e^{-2\pi i t\rho}J_{\frac{n-2}{2}+k}(2\pi r\rho)\hat{c}_{k,l}(\rho)\rho^{\frac{n}{2}}d\rho \big\|^2_{L_t^2 L_x^q}\Bigr)^{\frac 1 2}
\end{split}\end{equation}
Following \cite{Ste} we expand $\hat{c}_{k,l}(\rho) = \sum_{\nu\in \Z} c^{\nu}_{k,l}e^{i\frac{\pi}{2}\nu\rho}$ which upon substitution in the preceding formula leads to 
\begin{align*}
&r^{\frac{2-n}{2}}b_{k,l}(\theta)\cdot\int_0^\infty e^{-2\pi i t\rho}J_{\frac{n-2}{2}+k}(2\pi r\rho)\hat{c}_{k,l}(\rho)\rho^{\frac{n}{2}}d\rho\\
& = \alpha_{k}b_{k,l}(\theta)\sum_{\nu}r^{\frac{2-n}{2}}c^{\nu}_{k,l}\psi^k_{t-\frac{\nu}{4}}(r), 
\end{align*}
where we use 
\[
\psi^k_{t-\frac{\nu}{4}}(r) = \int_0^\infty J_{\frac{n-2}{2}+k}(2\pi r\rho) e^{-2\pi i(t-\frac{\nu}{4})\rho}\chi(\rho)\,d\rho,
\]
with $\chi$ a suitable smooth bump function localizing around the support of $\hat{f}(\rho\omega)$ with respect to $\rho$. But then from \cite{Ste} (see identities (82), (83) in loc. cit.) we have the bound 
\begin{align*}
&\big|\sum_{\nu}r^{\frac{2-n}{2}}c^{\nu}_{k,l}\psi^k_{t-\frac{\nu}{4}}(r)\big|\\&\lesssim \sum_{\nu}\frac{c^{\nu}_{k,l}}{(1+|t - \frac{\nu}{4}|)^{\frac{n-1}{2}}(1+\big| r- |t-\frac{\nu}{4}|\big|)^{\frac12}}\big[\frac{1}{(1+\big| r- |t-\frac{\nu}{4}|\big|)^{\frac12}} + R(k, |t-\frac{\nu}{4}|, r)\big],\\
\end{align*}
with $\sum_{\nu}\frac{1}{(1+\big| r- |t-\frac{\nu}{4}|\big|)}R^2(k, |t-\frac{\nu}{4}|, r)\lesssim 1$, and so application of the Cauchy-Schwarz inequality leads to the bound
\begin{align*}
\big\|\sum_{\nu}r^{\frac{2-n}{2}}c^{\nu}_{k,l}\psi^k_{t-\frac{\nu}{4}}(r)\big\|_{L^\infty( \mathbb{R}^+)}\lesssim \big(\sum_{\nu}\frac{|c^{\nu}_{k,l}|^2}{(1+|t - \frac{\nu}{4}|)^{n-1}}\big)^{\frac12}.
\end{align*}
Interpolating this with the simple energy bound 
\begin{equation*}
\big\|\big(\big|r^{\frac{2-n}{2}}\int_0^\infty e^{-2\pi i t\rho}J_{\frac{n-2}{2}+k}(2\pi r\rho)\hat{c}_{k,l}(\rho)\rho^{\frac{n}{2}}d\rho\big|^2\big)^{\frac12}\big\|_{L^2( \mathbb{R}^+, r^{n-1} dr)}\lesssim \big(\sum_{\nu}|c_{k,l}^{\nu}|^2\big)^{\frac12}, 
\end{equation*}
we find the bound 
\begin{equation}\label{eq:3}\begin{split}
&\big\| r^{\frac{2-n}{2}}b_{k,l}(\theta)\cdot\int_0^\infty e^{-2\pi i t\rho}J_{\frac{n-2}{2}+k}(2\pi r\rho)\hat{c}_{k,l}(\rho)\rho^{\frac{n}{2}}d\rho \big\|^2_{L^q( \mathbb{R}^+, r^{n-1} dr)}\\
&\lesssim \sum_{\nu}\frac{|c^{\nu}_{k,l}|^2}{(1+|t - \frac{\nu}{4}|)^{1+\epsilon_q}}.
\end{split}\end{equation}
for any $q$ such that $(2,q)$ is admissible in the sense of \eqref{eq:radialStr}, with $\epsilon_q>0$ a suitable positive number.

Keeping in mind that $b_{k,l}$ is a good frame satisfying \eqref{eq:bbound}, we then infer that
\begin{align*}
&\big\|\big(\sum_{k,l}\big|r^{\frac{2-n}{2}}b_{k,l}(\theta)\cdot\int_0^\infty e^{-2\pi i t\rho}J_{\frac{n-2}{2}+k}(2\pi r\rho)\hat{c}_{k,l}(\rho)\rho^{\frac{n}{2}}d\rho\big|^2\big)^{\frac12}\big\|_{L_x^q}\\
&\lesssim \big(\sum_{k,l}\sum_{\nu}\frac{|c^{\nu}_{k,l}|^2}{(1+|t - \frac{\nu}{4}|)^{n-1}}\big)^{\frac12}.
\end{align*}

Keeping in mind \eqref{eq:1} and substituting the preceding bound, we infer that 
\begin{equation}\label{eq:4}\begin{split}
&\big\|u^{(\tilde{\omega})}\big\|_{L^s_{\tilde{\omega}}L_t^2 L_x^q}\lesssim_q\sqrt{s}\big(\sum_{\nu} |c^{\nu}_{k,l}|^2\big)^{\frac12}\lesssim \sqrt{s}\|  f\big\|_{L^2(\R^n)}, s\geq q. 
\end{split}\end{equation}

The proposition is then a consequence of lemma 4.5 in \cite{Tzve}. 

\end{proof}
\subsubsection{A non pinching condition}
In this section we show how we can avoid the choice of a particular frame and work directly in an arbitrary eigenbasis of spherical harmonics. We start as previously and consider a function $f(x)$ on $\R^n$ supported at frequency $\sim 1$, and write its Fourier transform in terms of an arbitrary frame after passage to spherical coordinates $\rho, \omega$: 
\[
\hat{f}(\rho\omega) = \sum_{k,l}\hat{c}_{k,l}(\rho)b_{k,l}(\theta)
\]
In turn, this gives a representation of $f(x)$ in terms of the  basis by~\eqref{rep} with~\eqref{rep2}. We now assume that the decomposition~\eqref{rep} satisfies the following {\em non pinching } condition (see~\cite[(1.3]{PoRoTh})
\begin{assum}\label{asum.1}There exists $C>0$ such that for   any $k$ the projection 
$$\Pi_k (\hat{f}(\rho\omega)) = \sum_{l}\hat{c}_{k,l}(\rho)b_{k,l}(\theta)$$ on the $N_k$ dimensional space spanned by the spherical harmonics of degree $k$ satisfies
$$\forall \nu \in \mathbb{Z},  \forall k,  l; |l| \leq N_k  , | \hat{c}^\nu_{k,l}|^2  \leq \frac{ C} {N_k} \sum_l | \hat{c}^\nu_{k,l}|^2, \qquad  \hat{c}_{k,l}(\rho) = \sum_{\nu\in \Z} c^{\nu}_{k,l}e^{i\frac{\pi}{2}\nu\rho}$$
\end{assum}
We now randomize the function $f$ using the exact same procedure as in Section~\ref{sec.1.2.1}.
We now have 
\begin{prop}\label{prop:randomStrich1-bis} Let $(2, q)$ be admissible in the sense of \eqref{eq:radialStr}. Then for suitable positive constants $c, C$ we have under this new randomization the same results as in Proposition~\ref{prop:randomStrich1}
\begin{equation}
\mathbf{P}\big(\{\big\|u^{(\tilde{\omega})}\big\|_{L_t^2 L_x^q(\R^{n+1})}>\lambda\}\big)\leq \begin{cases} C e^{-c\frac{\lambda^2}{\| f\|_{L^2(\R^n)}^2}} &\text{ if } q<+\infty\\
C e^{-c\frac{\lambda^2}{\|f\|_{H^{0+}(\R^n)}^2}} &\text{ if } q=+\infty
\end{cases}
\end{equation}
\end{prop}
\begin{proof}
We revisit the proof of Proposition~\ref{prop:randomStrich1} and get from~\eqref{eq:1},~\eqref{eq:3}
\begin{equation}\label{eq:1bis}\begin{split}
&\big\|u^{(\tilde{\omega})}\big\|_{L^s_{\tilde{\omega}}L_t^2 L_x^q} \leq C\sqrt{s}\Big(\big\| \Bigl(\sum_{\nu}\sum_{k,l}\frac{|c^{\nu}_{k,l}|^2}{(1+|t - \frac{\nu}{4}|)^{1+\epsilon_q}}|b_{k,l} (\omega)|^2 \Bigr)^{1/2} \big\|^2_{L_t^2 L_{\theta}^q}\Bigr)^{\frac 1 2}\\
&\leq C\sqrt{s} \Big(\big\| \Bigl(\sum_{\nu}\sum_{k} \frac{ \max_{l} |c^{\nu}_{k,l}|^2}{(1+|t - \frac{\nu}{4}|)^{1+\epsilon_q}} \sum_{l } |b_{k,l} (\omega)|^2 \Bigr)^{1/2} \big\|^2_{L_t^2 L_{\omega}^q}\Bigr)^{\frac 1 2}
\end{split}\end{equation}
Following~\cite[Lemme 3.1]{BurqLeb} we now remark that 
$$(\theta, \widetilde \theta)\in (\mathbb{S}^{n-1})^2 \mapsto K_k( \theta, \widetilde \theta) = \sum_l b_{k,l} ( \theta), \overline{ b_{k,l}}(\widetilde \theta) $$
is the kernel of the spectral projector on $E_k$ the subspace of $L^2 ( \mathbb{S}^{n-1})$ spanned by the spherical harmonics of degree $k$. It is consequently invariant by conjugations by isometries of the sphere, which means, for any such isometry $J$
$$ K_k(J\theta , J\widetilde \theta) = K_k (\theta, \widetilde \theta),$$
which implies (since the group of isometries acts transitively on the sphere) that the function $\omega \mapsto K_k ( \omega, \omega)$ is constant on the sphere with mean value equal to 
$$ \sum_{l} \| b_{k,l}\|_{L^2}^2= N_k,$$ which implies
$$ \forall k, \sum_{l} | b_{k,l}(\theta)|^2 = \frac{ N_k} {\text{Vol} ( \mathbb{S}^{n-1})}.$$
 Plugging this into the r.h.s. of~\eqref{eq:1bis}  (remark that since what we get does not depend on $\omega$ any more, the $L^q_\omega$ norm becomes irrelevant) and using Assumption~\ref{asum.1} gives  
\begin{multline}\label{eq:1ter}
\big\|u^{(\tilde{\omega})}\big\|_{L^s_{\tilde{\omega}}L_t^2 L_x^q} \leq C\sqrt{s} \Big(\big\| \Bigl(\sum_{\nu}\sum_{k} \frac{ N_k \max_{l} |c^{\nu}_{k,l}|^2}{(1+|t - \frac{\nu}{4}|)^{1+\epsilon_q}}  \Bigr)^{1/2} \big\|^2_{L_t^2}\Bigr)^{\frac 1 2}\\
\leq C\sqrt{s} \Big(\big\| \Bigl(\sum_{\nu}\sum_{k} \frac{ C\sum_{l} |c^{\nu}_{k,l}|^2}{(1+|t - \frac{\nu}{4}|)^{1+\epsilon_q}}  \Bigr)^{1/2} \big\|^2_{L_t^2}\Bigr)^{\frac 1 2}\leq C \sqrt{s} \Bigl(\sum_{\nu,k,l} |c_{k,l}^\nu|^2\Bigr)^{1/2}.
\end{multline}
The proposition is then a consequence of lemma 4.5 in \cite{Tzve}. 
\end{proof}

\subsection{A refinement; microlocalized Strichartz estimates}

For applications of the estimates derived in the preceding subsection, and in particular for deriving estimates {\it{which beat the natural scaling}}, it is useful to also control certain square sums over pieces which are box-localised in Fourier space. 
Specifically, it shall be useful to control norms of the form 
\[
\big(\sum_{c}\big\|P_c u^{(\tilde{\omega})}(t, x)\big\|_{L_t^p L_x^q}^2\big)^{\frac12},
\]
where $c$ ranges over a covering of the annulus $\rho\sim 1$ in Fourier space by boxes of diameter $\sim \mu\lesssim 1$. Here, after re-scaling to frequency $\sim 2^k$, $k\gg 1$, we shall put $\mu = 1$. The point here shall be to deduce a bound which beats the 'trivial' estimates obtained by Cauchy-Schwarz and interpolation. 
In order to achieve the optimal such bound, we shall have to implement another randomisation, this time with respect to the radial direction. Specifically, divide the interval $\rho\sim 1$ into subintervals $I$ of length $\sim \mu$. Then, for each $k,l$, write 
\[
\hat{c}_{k,l}(\rho) = \sum_{I}\chi_I(\rho)\hat{c}_{k,l}(\rho) =:\sum_I \hat{c}_{k,l}^{(I)}(\rho)
\]
For each $I$ write 
\begin{align*}
\hat{c}_{k,l}^{(I)}(\rho) = \sum_{\nu\in Z}c_{k,l}^{(I),\nu}e^{i\frac{\pi}{2}\nu\rho}
\end{align*}
Finally, let $h_{k,l}^{(I), \nu}(\tilde{\omega}_1)$ be a family of independent random variables on a probability space $\Omega_1$, and consider the randomised functions 
\[
\sum_{\nu\in Z}c_{k,l}^{(I),\nu}h_{k,l}^{(I), \nu}(\tilde{\omega}_1)e^{i\frac{\pi}{2}\nu\rho},
\]
and so we replace $\hat{c}_{k,l}^{(I)}(\rho)$ by 
\[
\sum_I\sum_{\nu\in Z}c_{k,l}^{(I),\nu}h_{k,l}^{(I), \nu}(\tilde{\omega}_1)e^{i\frac{\pi}{2}\nu\rho}.
\]
Call the resulting free wave $u^{(\tilde{\omega}, \tilde{\omega}_1)}$. 
Below, we shall denote by $\mathbf{P}$ the probability space $\Omega\times \Omega_1$.

\begin{prop}\label{prop:randomStrich2} Let $f, u^{(\tilde{\omega},\tilde{\omega}_1)}$ be as in the Section~\ref{sec.1.2.1}. Also, let $0<\mu\lesssim 1$ be a length scale, and pick for such $\mu$ a uniformly finitely overlapping cover $\mathcal{C}$ of the annulus $\rho\sim 1$ by cubes $c$ of diameter $\mu$. Denote by $P_c$ a Fourier multiplier which localizes smoothly to $c$. Then for $q =\frac{2(n-1)}{n-2}+$, we have 
\begin{align*}
\mathbf{P}\big(\{\big(\sum_{c\in\mathcal{C}}\big\|P_c u^{(\tilde{\omega}, \tilde{\omega}_1)}\big\|_{L_t^2 L_x^q}^2\big)^{\frac12}>\lambda\mu^{-\frac{n-2}{2(n-1)}-}\}\big)\lesssim De^{-d{\lambda^2}/{\| f\|_{L^2(\R^n)}^2}}
\end{align*}
for suitable positive constants $D, d$ (which, in addition to the implicit constant, depend on $q$).
\end{prop}
\begin{proof} We follow a similar procedure as in the preceding proof. Let $\check{\chi}_{c}(x)$ be the inverse Fourier transform of the smooth localiser $\chi_{c}$ which realises $P_c$. Note that $\check{\chi}_{c}$ rapidly decays beyond scale $\mu^{-1}$, and we have $\big\|\check{\chi}_{c}\big\|_{L^\infty}\lesssim \mu^n$. 
\\
Picking $s\geq q$, we have by Minkowski's inequality
\begin{equation}\label{eq:5}
\big\|\big\|\{\big\|P_c u^{(\tilde{\omega},\tilde{\omega}_1)}\big\|_{L_t^2 L_x^q}\}\big\|_{l^2_c}\big\|_{L^s_{\tilde{\omega},\tilde{\omega}_1}}\leq \big\|\{\big\|\big\|P_c u^{(\tilde{\omega},\tilde{\omega}_1)}\big\|_{L^s_{\tilde{\omega},\tilde{\omega}_1}}\big\|_{L_t^2 L_x^q}\}\big\|_{l^2_c}
\end{equation}
On the other hand, observe that we can write
\begin{align*}
&P_c u^{(\tilde{\omega}, \tilde{\omega}_1)}(t, r\omega)\\& = P_c\big(\sum_{k,l, \nu}2\pi i^kr^{\frac{2-n}{2}}h_{k,l}(\tilde{\omega})b_{k,l}(\theta)\cdot\int_0^\infty e^{-2\pi i (t - \frac{\nu}{4})\rho}J_{\frac{n-2}{2}+k}(2\pi r\rho)c_{k,l}^{(I(c)), \nu}h_{k,l}^{(I(c)), \nu}(\tilde{\omega}_1)\rho^{\frac{n}{2}}d\rho\big), 
\end{align*}
where $I(c)$ is an interval of length $\sim \mu$ essentially uniquely associated with $c\in \mathcal{C}$. Carrying out the integral, we find 
\begin{align*}
&P_c u^{(\tilde{\omega}, \tilde{\omega}_1)}(t, r\omega) \sim \sum_{k,l, \nu}h_{k,l}(\tilde{\omega})h_{k,l}^{(I(c)), \nu}(\tilde{\omega}_1)c_{k,l}^{(I(c)), \nu}\check{\chi}_{c}*\big[b_{k,l}(\theta)r^{\frac{2-n}{2}}\psi^k_{t-\frac{\nu}{4}}(r)\big]
\end{align*}
where $\sim $ indicates 'up to an irrelevant constant'. We conclude that for any $s\geq q$ we have 
\begin{align*}
\big\|P_c u^{(\tilde{\omega},\tilde{\omega}_1)}\big\|_{L^s_{\tilde{\omega},\tilde{\omega}_1}}\leq D \sqrt{s} \big(\sum_{k,l, \nu}\big|c_{k,l}^{(I(c)), \nu}\big|^2 \big|\check{\chi}_{c}*\big[b_{k,l}(\theta)r^{\frac{2-n}{2}}\psi^k_{t-\frac{\nu}{4}}(r)\big]\big|^2\big)^{\frac12}
\end{align*}
It follows that in order to bound the right hand side of \eqref{eq:5}, we need to bound 
\begin{multline*}
\big\|\big(\sum_c\big\|\big(\sum_{k,l, \nu}\big|c_{k,l}^{(I(c)), \nu}\big|^2 \big|\check{\chi}_{c}*\big[b_{k,l}(\theta)r^{\frac{2-n}{2}}\psi^k_{t-\frac{\nu}{4}}(r)\big]\big|^2\big)^{\frac12}\big\|_{L_x^q}^2\big)^{\frac12}\big\|_{L_t^2}, \\
\lesssim\big\| \bigl( \sum_{k,l,\nu. I} \sum_{c; I(c)=I} \| c_{k,l}^{I, \nu} \check{\chi}_{c}*\big[b_{k,l}(\theta)r^{\frac{2-n}{2}}\psi^k_{t-\frac{\nu}{4}}(r)\big]\big\|_{L_x^q}^2\bigr)^{1/2} \|_{L^2_t}
\end{multline*}
and more specifically the inner expression without the outer norm $\|\cdot\|_{L_t^2}$. This we shall achieve as in \cite{Ste} via interpolation between bounds for $q = \infty$ and $q = 2$. 
Then using the same point wise bounds as before, we find that for any function $g$ on the sphere $\mathbb{S}^{n-1}$ 
\begin{align*}
 \big|\check{\chi}_{c}*\big[g(\omega)r^{\frac{2-n}{2}}\psi^k_{t-\frac{\nu}{4}}(r)\big]\big|&\lesssim \| g\|_{L^\infty}\big|\check{\chi}_{c}\big|*\big[\frac{1}{(1+|t-\frac{\nu}{4}|)^{\frac{n-1}{2}}}\frac{1}{(1+\big|r - |t-\frac{\nu}{4}|\big|)^{\frac12}}\big]\\
 &\lesssim   \| g\|_{L^\infty}\mu^n\cdot \mu^{-(n-\frac12)}\cdot\frac{1}{(1+|t-\frac{\nu}{4}|)^{\frac{n-1}{2}}}
\end{align*}
from which we deduce the bound 
$$\sup_{c; I(c) =I} \|\check{\chi}_{c}*\big[g(\omega)r^{\frac{2-n}{2}}\psi^k_{t-\frac{\nu}{4}}(r)\big]\|_{L^\infty}|\lesssim   \| g\|_{L^\infty} \mu^{\frac12}\cdot\frac{1}{(1+|t-\frac{\nu}{4}|)^{\frac{n-1}{2}}}
$$
The trivial $L^2$ bound 
$$\bigl(\sum_{c; I(c) =I} \|\check{\chi}_{c}*\big[g(\omega)r^{\frac{2-n}{2}}\psi^k_{t-\frac{\nu}{4}}(r)\big]\|_{L^2}^2\bigr)^{1/2}\lesssim   \| g\|_{L^2}
$$
and interpolation gives for $q = \frac{2(n-1)}{n-2}+$
\begin{equation*}
\bigl(\sum_{c; I(c) =I} \big\|\check{\chi}_{c}*\big[g(\omega)r^{\frac{2-n}{2}}\psi^k_{t-\frac{\nu}{4}}(r)\big]\big\|_{L^q}^q\bigr)^{1/q}\lesssim   \| g\|_{L^q( \mathbb{S}^{n-1})}\mu^{\frac{1} {2(n-1)}-}\cdot\frac{1}{(1+|t-\frac{\nu}{4}|)^{\frac{1}{2}+}}.
\end{equation*}
Finally, H\"older inequality and the uniform bound on the $L^q$ norm of $b_{k,l}$ from Proposition~\ref{prop:randomStrich1} gives 
\begin{equation}\label{linfty}
\begin{aligned}
&\big(\sum_c\big\|\big(\sum_{k,l, \nu}\big|c_{k,l}^{(I(c)), \nu}\big|^2 \big|\check{\chi}_{c}*\big[b_{k,l}(\theta)r^{\frac{2-n}{2}}\psi^k_{t-\frac{\nu}{4}}(r)\big]\big|^2\big)^{\frac12}\big\|_{L_x^q}^2\big)^{\frac12}\\
&=\big( \sum_I \sum_{c; I(c) =I}  \sum_{k,l, \nu}\| c_{k,l}^{(I, \nu}\check{\chi}_{c}*\big[b_{k,l}(\theta)r^{\frac{2-n}{2}}\psi^k_{t-\frac{\nu}{4}}(r)\big]\big)\big\|^2_{L_x^q}\big)^{\frac12}\\
&\lesssim \Big( \sum_{k,l, \nu, I} \bigl(\sum_{c; I(c) =I} \big\|c_{k,l}^{I, \nu} \check{\chi}_{c}*\big[b_{k,l}(\theta)r^{\frac{2-n}{2}}\psi^k_{t-\frac{\nu}{4}}(r)\big]\big)\big\|^q_{L_x^q}\big)^{\frac2q} \mu^{-1} \Bigr)^{1/2} \\
&\lesssim\mu^{-\frac{n-2}{2(n-1)}-}\big(\sum_{k,l, \nu, I}  |c_{k,l}^{(I),\nu}|^2 \frac{1}{(1+|t-\frac{\nu}{4}|)^{1+}}\big)^{\frac12},
\end{aligned}
\end{equation}
In total, we infer for such $q$ the bound 
\begin{align*}
&\big\|\big(\sum_c\big\|\big(\sum_{k,l, \nu}\big|c_{k,l}^{(I(c)), \nu}\big|^2 \big|\check{\chi}_{c}*\big[b_{k,l}(\theta)r^{\frac{2-n}{2}}\psi^k_{t-\frac{\nu}{4}}(r)\big]\big|^2\big)^{\frac12}\big\|_{L_x^q}^2\big)^{\frac12}\big\|_{L_t^2}\\
&\lesssim \mu^{-\frac{n-2}{2(n-1)}-}\big(\sum_{k,l, \nu, I}  |c_{k,l}^{(I),\nu}|^2\big)^{\frac12},
\end{align*}
and so 
\begin{align*}
\big\|\big\|\{\big\|P_c u^{(\tilde{\omega},\tilde{\omega}_1)}\big\|_{L_t^2 L_x^q}\}\big\|_{l^2_c}\big\|_{L^s_{\tilde{\omega},\tilde{\omega}_1}}\leq D \sqrt{s}\mu^{-\frac{n-2}{2(n-1)}-}\big\|f\big\|_{L_x^2}
\end{align*}
The proposition is a consequence of this via Lemma 4.5 in \cite{Tzve}. 
\end{proof}

\subsection{Comparison to the Klainerman-Tataru improved Strichartz estimate}

Recall from \cite{KlTat} that in dimension $n\geq 4$, we have the following bound for free waves $u$ supported at frequency $\sim 1$
\begin{equation}\label{eq:KlTat}
\big(\sum_{c\in \mathcal{C}}\big\| P_c u\big\|_{L_t^2 L_x^p(\R^{n+1})}^2\big)^{\frac12}\lesssim \mu^{\frac12 - \frac{1}{p}}\big\|u[0]\big\|_{L_x^2},
\end{equation}
provided $(2,p)$ is (standard) Strichartz admissible. Using the endpoint exponent $p = \frac{2(n-1)}{n-3}$, and using Bernstein's inequality, we can improve this for $p = \infty$ to 
\begin{align*}
\big(\sum_{c\in \mathcal{C}}\big\| P_c u\big\|_{L_t^2 L_x^\infty(\R^{n+1})}^2\big)^{\frac12}&\lesssim \mu^{\frac12 - \frac{n-3}{2(n-1)} + \frac{n(n-3)}{2(n-1)}}\big\|u[0]\big\|_{L_x^2}\\
&\lesssim \mu^{\frac{n-2}{2}}\big\|u[0]\big\|_{L_x^2}.\\
\end{align*}
On the other hand, assuming for $q = \frac{2(n-1)}{n-2}+$ the bound 
\[
\big(\sum_{c\in \mathcal{C}}\big\| P_c u\big\|_{L_t^2 L_x^p(\R^{n+1})}^2\big)^{\frac12}\lesssim \mu^{-\frac{n-2}{2(n-1)}-}\big\| u[0]\big\|_{L_x^2}
\]
leads via Bernstein's inequality to the bound 
\begin{align*}
\big(\sum_{c\in \mathcal{C}}\big\| P_c u\big\|_{L_t^2 L_x^\infty(\R^{n+1})}^2\big)^{\frac12}&\lesssim \mu^{-\frac{n-2}{2(n-1)}-}\mu^{\frac{n(n-2)}{2(n-1)}}\big\|u[0]\big\|_{L_x^2}\\
& = \mu^{\frac{n-2}{2}-}\big\|u[0]\big\|_{L_x^2},\\
\end{align*}
which is essentially compatible with the Klainerman-Tataru bound. However, the range of exponents $q$ in Proposition~\ref{prop:randomStrich2} is of course much larger than the one in \cite{KlTat}. 

\section{Small data global existence for the critical nonlinear wave equation in $n = 3$ dimensions with supercritical data}

\subsection{Some notational conventions}

In the sequel, we shall denote dyadic frequencies by $N = 2^k$, $k\in \Z$, and the associated standard Littlewood-Paley multipliers by $P_N$ or also $P_k$. For each $l>0$, we pick a uniformly finitely overlapping cover $K_l$ of $S^2$ by caps $\kappa$ of diameter $\sim 2^{-l}$, and denote the Fourier localizers which smoothly localise to frequency $\sim 2^k$ and angular sector $\kappa$ by $P_{k,\kappa}$. If $u$ is a function of $(t, x)$, we denote its restriction to $\tau><0$ (Fourier variables) by $Q^{\pm}u$ or $u^{\pm}$. We denote by $\big||\tau| - |\xi|\big|$ the {\it{modulation}}, and by $Q_j$ the multiplier which smoothly localises to modulation $\sim 2^j$. To define the spaces in the next section, we shall refer to null-frames $(t_{\omega}, x_{\omega^{\perp}})$, $\omega\in S^2$, which refer to $\frac{1}{\sqrt{2}}(t + \omega\cdot x)$ as well as $x - t_{\omega}\cdot \frac{1}{\sqrt{2}}\cdot(1,\omega)$. 
We shall frequently resort to {\it{Bernstein's inequality}}: for us this means the fact that for $p<q$ and $f\in L^p(\R^n)$ with Fourier support contained in a rectangular box $R$ we have 
\[
\big\|f\big\|_{L^q(\R^n)}\lesssim |R|^{\frac{1}{p} - \frac{1}{q}}\cdot \big\|f\big\|_{L^p(\R^n)}. 
\]

\subsection{Smoothness gains via Wiener randomisation}
Here we combine the preceding considerations with the Wiener randomisation introduced by Luhrmann-Mendelson in \cite{Luhr1}. We shall henceforth work in $n = 3$ spatial dimensions. Consider a datum $f(\rho\omega)$.  Write this as a sum of frequency localised pieces: 
\[
f = \sum_{N\geq 1} P_Nf =: f_{<0}+ \sum_N f^{(N)}.
\]
where $N$ ranges over dyadic numbers and $P_N$ is the standard Littlewood-Paley projector. To simplify things a bit, we shall assume $f_{<0} = 0$ in the sequel. We  randomise each component $f^{(N)}$ as in the last subsection but one, i. e. introduce $f^{(N), (\tilde{\omega}^{(N)}, \tilde{\omega}_1^{(N)})}$, where the superscripts in $\tilde{\omega}^{(N)}, \tilde{\omega}_1^{(N)}$ indicate that we randomise these pieces independently, of course. Introducing the corresponding propagators 
\[
u^{(N), (\tilde{\omega}^{(N)}, \tilde{\omega}^{(N)}_1)}(t, x) = \big(e^{-it\sqrt{-\triangle}}f^{(N), (\tilde{\omega}, \tilde{\omega}^{(N)}_1)}\big)(x), 
\]
and letting $\mathcal{C}$ be a finitely overlapping covering the frequency region $\rho \sim N$ by cubes if diameter $\mu\sim 1$, we have the following re-scaled version of the inequality of Proposition~\ref{prop:randomStrich2}, keeping in mind that we set $n = 3$: for any $q>4$, 
\begin{equation}\label{eq:keyinequality1}
\mathbf{P}\big(\{\big(\sum_{c\in\mathcal{C}}\big\|P_c u^{(N), (\tilde{\omega}^{(N)}, \tilde{\omega}_1^{(N)})}\big\|_{L_t^2 L_x^q(\R^{3+1})}^2\big)^{\frac12}>\lambda N^{\frac{1}{4}+}\}\big)\leq De^{-d\frac{\lambda^2}{\|f^{(N)}\|_{H^{\frac14+}(\R^3)}^2}}
\end{equation}
Alternatively, we get 
\begin{equation}\label{eq:keyinequality2}
\mathbf{P}\big(\{\big(\sum_{c\in\mathcal{C}}\big\|P_c u^{(N), (\tilde{\omega}^{(N)}, \tilde{\omega}_1^{(N)})}\big\|_{L_t^2 L_x^q(\R^{3+1})}^2\big)^{\frac12}>\lambda \}\big)\leq De^{-d\frac{\lambda^2}{\|f^{(N)}\|_{H^{\frac12+}(\R^3)}^2}}
\end{equation}

Assume now that 
\[
\big\|f\big\|_{H^{\frac12+}(\R^3)}\sim \big(\sum_{N\geq 1}\big\|f^{(N)}\big\|_{H^{\frac12+}(\R^3)}^2\big)^{\frac12}<\epsilon_*\ll 1.
\]
Then letting $\prod_{N\geq 1}(\Omega^{(N)}\times \Omega_1^{(N)})$ be the corresponding product probability space, we have 
\begin{align*}
&\mathbf{P}\big(\{\exists N\geq 1; \big(\sum_{c\in\mathcal{C}}\big\|P_c u^{(N), (\tilde{\omega}^{(N)}, \tilde{\omega}_1^{(N)})}\big\|_{L_t^2 L_x^q(\R^{3+1})}^2\big)^{\frac12}>\sqrt{\epsilon_*}\langle\log N\rangle\}\big)
\\
&\leq \sum_{N\geq 1} De^{-d\frac{\langle\log N\rangle^2\epsilon_*}{\|f^{(N)}\|_{H^{\frac12+}(\R^3)}^2}}\leq De^{-d\epsilon_*^{-1}}
\end{align*}
Replacing $\epsilon_*$ by $\langle\log N\rangle^6\epsilon_*$ and incorporating the correction into $\|f^{(N)}\|_{H^{\frac12+}(\R^4)}$, we see that  up to a data set of exponentially vanishing size $De^{-d\epsilon_*^{-1}}$, we may assume that for each dyadic $N$ we have the bound 
\begin{equation}\label{eq:keybound1}
\big(\sum_{c\in\mathcal{C}}\big\|P_c u^{(N), (\tilde{\omega}^{(N)}, \tilde{\omega}_1^{(N)})}\big\|_{L_t^2 L_x^q(\R^{3+1})}^2\big)^{\frac12}<\frac{\sqrt{\epsilon_*}}{\langle\log N\rangle^2}
\end{equation}
In order to take advantage of this bound, we now effect a third, final randomisation, this time at the scale of cubes of size $\sim 1$ covering frequency space. This is in effect exactly the procedure in \cite{Luhr1}. Thus for the usual random variable $h_c(\tilde{\omega}_3)$, where $c\in \mathcal{C}$ ranges over a collection of finitely overlapping cubes of diameter $\sim 1$ and $P_c$ the corresponding Fourier localizer, we consider 
\[
f^{(\tilde{\omega}^{*}, \tilde{\omega}_1^{*}, \tilde{\omega}_3)}: = \sum_N h_c(\tilde{\omega}_3) P_cf^{(N), (\tilde{\omega}^{(N)}, \tilde{\omega}_1^{(N)})}
\]
Call the corresponding propagator 
\[
u^{(\tilde{\omega}^{*}, \tilde{\omega}_1^{*}, \tilde{\omega}_3)}: = \big(e^{-it\sqrt{-\triangle}}f^{(\tilde{\omega}^{*}, \tilde{\omega}_1^{*}, \tilde{\omega}_3)}\big)(x). 
\]
If $\tilde{\omega}_3$ is defined on probability space $\Omega_3$, with probability measure $\mathbf{P}_3$, then a combination of Bernstein's inequality with \eqref{eq:keybound1} furnishes the following 
\begin{lem}\label{lem:twoinfty} Assume that $\big\|f\big\|_{H^{\frac{1}{2}+}(\R^3)}<\epsilon_*$ and that $(\tilde{\omega}^{*}, \tilde{\omega}_1^{*})$ avoids an exceptional set of measure $\lesssim e^{-\frac{d}{\epsilon_*}}$. Then we have for any $M\in (4,\infty)$
\[
\mathbf{P}_3\big(\{\big(\sum_{N\geq 1}\langle\log N\rangle^2\big\| P_N u^{(\tilde{\omega}^{*}, \tilde{\omega}_1^{*}, \tilde{\omega}_3)}\big\|_{L_t^2 L_x^M(\R^{3+1})}^2\big)^{\frac12}>\lambda\}\big)< G e^{-g\frac{\lambda^2}{\epsilon_*}}
\]
In particular, up to a set of parameters $\tilde{\omega}_3$ of size $<e^{-\frac{g}{\sqrt{\epsilon_*}}}$, we have 
\begin{align*}
\big(\sum_{N\geq 1}\langle\log N\rangle^2\big\| P_N u^{(\tilde{\omega}^{*}, \tilde{\omega}_1^{*}, \tilde{\omega}_3)}\big\|_{L_t^2 L_x^M(\R^{3+1})}^2\big)^{\frac12}<\epsilon_*^{\frac14},
\end{align*}
which implies 
\begin{align*}
\sum_{N\geq 1}\big\| P_N u^{(\tilde{\omega}^{*}, \tilde{\omega}_1^{*}, \tilde{\omega}_3)}\big\|_{L_t^2 L_x^M(\R^{3+1})}\lesssim \epsilon_*^{\frac14},
\end{align*}
\end{lem}
\begin{proof}
We have for any $s\geq M$ the bound 
\begin{align*}
&\big\|\big(\sum_{N\geq 1}\langle\log N\rangle^2\big\| P_N u^{(\tilde{\omega}^{*}, \tilde{\omega}_1^{*}, \tilde{\omega}_3)}\big\|_{L_t^2 L_x^M(\R^{3+1})}^2\big)^{\frac12}\big\|_{L_{\tilde{\omega}_3}^s}\\&\leq G\sqrt{s}\big(\sum_{N\geq 1}\sum_{c\in\mathcal{C}}\langle\log N\rangle^2\big\| P_N P_c u^{(\tilde{\omega}^{*}, \tilde{\omega}_1^{*})}\big\|_{L_t^2 L_x^M(\R^{3+1})}^2\big)^{\frac12}\\
&\lesssim G\sqrt{s}\big(\sum_{N\geq 1}\sum_{c\in\mathcal{C}}\langle\log N\rangle^2\big\|P_c u^{(N), (\tilde{\omega}^{(N)}, \tilde{\omega}_1^{(N)})}\big\|_{L_t^2 L_x^q(\R^{3+1})}^2\big)^{\frac12}
\end{align*}
for any $4<q\leq M$. The assertion then follows from \eqref{eq:keybound1} and lemma 4.5 in \cite{Tzve}. 
\end{proof}
For later reference, we shall want to adapt this result to data of varying degrees of smoothness. We have 
\begin{lem}\label{lem:twoinftyrough} Assume that $\big\|f\big\|_{H^{s}(\R^3)}<\epsilon_*$, $s\in \R$, and that $(\tilde{\omega}^{*}, \tilde{\omega}_1^{*})$ avoids an exceptional set of measure $\lesssim e^{-\frac{d}{\epsilon_*}}$. Then we have for any $M\in (4,\infty)$
\[
\mathbf{P}_3\big(\{\big(\sum_{N\geq 1}\langle\log N\rangle^2 N^{2s-1-}\big\| P_N u^{(\tilde{\omega}^{*}, \tilde{\omega}_1^{*}, \tilde{\omega}_3)}\big\|_{L_t^2 L_x^M(\R^{3+1})}^2\big)^{\frac12}>\lambda\}\big)< G e^{-g\frac{\lambda^2}{\epsilon_*}}
\]
In particular, up to a set of parameters $\tilde{\omega}_3$ of size $<e^{-\frac{g}{\sqrt{\epsilon_*}}}$, we have 
\begin{align*}
\big(\sum_{N\geq 1}\langle\log N\rangle^2 N^{2s-1-}\big\| P_N u^{(\tilde{\omega}^{*}, \tilde{\omega}_1^{*}, \tilde{\omega}_3)}\big\|_{L_t^2 L_x^M(\R^{3+1})}^2\big)^{\frac12}<\epsilon_*^{\frac14}.
\end{align*}
By Bernstein's inequality, this implies (choosing $M$ sufficiently large) that 
\begin{align*}
\big(\sum_{N\geq 1}\langle\log N\rangle^2 N^{2s-1-}\big\| P_N u^{(\tilde{\omega}^{*}, \tilde{\omega}_1^{*}, \tilde{\omega}_3)}\big\|_{L_t^2 L_x^\infty(\R^{3+1})}^2\big)^{\frac12}<\epsilon_*^{\frac14}.
\end{align*}
\end{lem}

For later reference, we shall also need randomised bounds for the $L_{t,x}^\infty$-norm which beat the scaling. These can be obtained easily by using an $L_t^{M}L_x^{2+}$-Strichartz bound and invoking Bernstein's inequality to pass to an $L_{t,x}^\infty$-bound: 
\begin{lem}\label{lem:inftyinfty} Let $f$ be as in the preceding lemma. For fixed $(\tilde{\omega}^{*}, \tilde{\omega}_1^{*})$, there is a set of parameters $\tilde{\omega}_3$ of size $>1-e^{-\frac{g}{\sqrt{\epsilon_*}}}$, such that 
\begin{align*}
\big(\sum_{N\geq 1}\langle\log N\rangle^2 N^{2s-}\big\| P_N u^{(\tilde{\omega}^{*}, \tilde{\omega}_1^{*}, \tilde{\omega}_3)}\big\|_{L_{t,x}^\infty(\R^{3+1})}^2\big)^{\frac12}<\epsilon_*^{\frac14}.
\end{align*}
\end{lem}

\subsection{The model problem: a class a fractional Wave Maps type equations}

Now let $\alpha\in (0,1)$ and introduce the Riesz type operators $R_{\nu}^{(\alpha)}: = \partial_{\nu}|\nabla|^{-\alpha}$, $\nu = 0, 1, 2, 3$ where we put $|\nabla| = \sqrt{-\triangle}$. We also set $\partial^{\nu} = m^{\mu\nu}\partial_{\mu}$, where $m_{\mu\nu}$ denotes the Minkowski metric. Then consider the following class of equations on $\R^{3+1}$: 
\begin{equation}\label{eq:alphaMain}
\Box u = R_{\nu}^{(\alpha)}u R^{\nu,(\alpha)}u.
\end{equation}
Note that for $\alpha = 0$ this attains essentially the form of the Wave Maps equation. This problem scales according to $u(t, x)\longrightarrow \lambda^{2\alpha}u(\lambda t, \lambda x)$, and so the critical Sobolev space is $\dot{H}^{\frac32 - 2\alpha}$. 
Put $s_{\alpha}: = \frac32 - 2\alpha$. Then 
\begin{thm}\label{thm:alpharough} Let $\alpha\in (0,\frac14)$, and let $s>s_{\alpha} - \frac{\alpha}{2}$, and $u[0] = \big(e^{-it\sqrt{-\triangle}}f^{(\tilde{\omega}^{*}, \tilde{\omega}_1^{*}, \tilde{\omega}_3)}\big)[0]$, where 
\[
\big\|f\big\|_{H^{s}(\R^3)}<\epsilon_*,
\]
with $\epsilon_*>0$ sufficiently small. Then for all $(\tilde{\omega}^{*}, \tilde{\omega}_1^{*})$ avoiding an exceptional set of measure $\lesssim e^{-\frac{d}{\epsilon_*}}$, and all 
$\tilde{\omega}_3$ avoiding a set of size $<e^{-\frac{g}{\sqrt{\epsilon_*}}}$, the problem \eqref{eq:alphaMain} with data $u[0]$ admits a global solution which decouples into 
\[
u = e^{-it\sqrt{-\triangle}}f^{(\tilde{\omega}^{*}, \tilde{\omega}_1^{*}, \tilde{\omega}_3)} + v,
\]
where $v\in L_t^\infty \dot{H}^{s_\alpha}$, 
 \end{thm}
 \begin{rem}\label{rem:thm}Observe that when $\alpha =0$, the nonlinearity no longer has any smoothing effect, and achieving a supercritical well-posedness result will have to employ different techniques. 
 The present result appears to rely crucially on the improved range of Strichartz estimates due to Lemma~\ref{lem:twoinftyrough}.
  \end{rem}
  \begin{proof} Write $u_1 = e^{-it\sqrt{-\triangle}}f^{(\tilde{\omega}^{*}, \tilde{\omega}_1^{*}, \tilde{\omega}_3)}$. Then outside of exceptional parameter sets as in the statement of the theorem, we may assume that 
  \begin{equation}\label{eq:smallu1}
  \sum_N N^{s_{\alpha} -  \frac{\alpha}{2+}-\frac12}\big\|P_Nu_1\big\|_{L_t^2 L_x^\infty}\leq \epsilon_*^{\frac14},
  \end{equation}
   \begin{equation}\label{eq:smallu11}
  \big(\sum_N\sum_{c\in \mathcal{C}} [N^{s_{\alpha} -  \frac{\alpha}{2+}-\frac12}\big\|P_cP_{N}u_1\big\|_{L_t^2 L_x^\infty}]^2\big)^{\frac12}\leq \epsilon_*^{\frac14},
  \end{equation}
where now $\mathcal{C}$ is a covering of all of frequency space by cubes of diameter $\sim 1$. Repeating the argument in the preceding subsection, we may also assume that we have 
 \begin{equation}\label{eq:smallu111}
  \big(\sum_N\sum_{\kappa\in K_l} N^{s_{\alpha} -  \frac{\alpha}{2+}-\frac12}\big\|P_{\log N, \kappa}u_1\big\|_{L_t^2 L_x^\infty}^2\big)^{\frac12}\leq \epsilon_*^{\frac14},
  \end{equation}
provided $l\in [0,\log N]$. 

  Similarly, in light of Lemma~\ref{lem:inftyinfty}, we may assume 
  \begin{equation}\label{eq:smallu2}
  \sum_N N^{s_{\alpha} - \frac{\alpha}{2+}}\big\|P_Nu_1\big\|_{L_{t,x}^\infty}\leq \epsilon_*^{\frac14}. 
  \end{equation}
  Finally, recalling Lemma~\ref{lem:twoinfty}, we may assume for any fixed $M>4$ the bound 
    \begin{equation}\label{eq:smallu22}
  \sum_N N^{s_{\alpha} -  \frac{\alpha}{2+}-\frac12}\big\|P_N u_1\big\|_{L_t^2 L_x^M}\leq \epsilon_*^{\frac14}. 
  \end{equation}
  The equation for $v$ becomes schematically 
  \begin{equation}\label{eq:veqn}\begin{split}
  &\Box v = R_{\nu}^{(\alpha)}u_1 R^{\nu,(\alpha)}u_1 + R_{\nu}^{(\alpha)}u_1 R^{\nu,(\alpha)}v + R_{\nu}^{(\alpha)}v R^{\nu,(\alpha)}v,\\&v[0] = \big(v(0,\cdot), v_t(0,\cdot)\big) = (0,\,0).
  \end{split}\end{equation}
  We shall establish a fixed point here in a suitable space at regularity $\dot{H}^{s_{\alpha}}$. In light of the null-structure inherent in the nonlinearity, a variant of the norms used in \cite{Tao1} here works:
  \begin{equation}\label{eq:S_knorm}
  \|v\|_{S_k} = 2^{-k}\big(\|\nabla_{t,x}v\|_{\dot{X}^{s_{\alpha},\frac12,\infty}_k} + \|\nabla_{t,x}v\|_{S_k^{Str}} + \|\nabla_{t,x}v\|_{S_k^{ang}}\big),
  \end{equation}
  where we set (for scaling reasons) and assuming $p, s\in [1,\infty)$
  \begin{align*}
   &\|v\|_{S_k^{Str}}: = \bigcap_{\substack{\frac{1}{p} + \frac{1}{q}\leq \frac12\\ p>2+}}2^{-(s_\alpha +\frac{1}{p}+3(\frac{1}{q} - \frac12))k}L_t^p L_x^q,\\
   &\|v\|_{\dot{X}_k^{s,p,q}}: = 2^{s\cdot k}\big(\sum_{j\in \Z}[2^{p\cdot j}\|Q_j v\|_{L_{t,x}^2}]^q\big)^{\frac{1}{q}}\,1\leq q<\infty,\,\|v\|_{\dot{X}_k^{s,p,\infty}}: =  2^{s\cdot k}\sup_{j\in \Z}2^{p\cdot j}\|Q_j v\|_{L_{t,x}^2},
  \end{align*}
  and we use the following version of the null-frame spaces: 
  \begin{align*}
  \|v\|_{S_k^{ang}}: = \sup_{l>0}\big(\sum_{\kappa\in K_l}\big\|P_{k,\kappa}Q^{\pm}_{<k-2l}v\big\|_{S_{k,\pm\kappa}}^2\big)^{\frac12},
  \end{align*}
  where 
  \begin{align*}
  \big\|v\big\|_{S_{k,\kappa}}^2: = \big\|v\big\|_{S_{k}^{Str}}^2 + 2^{2k\cdot s_\alpha }\sup_{\omega\notin 2\kappa}\text{dist}(\omega, \kappa)^2\big\|v\|_{L^\infty_{t_\omega} L^2_{x_{\omega^\perp}}}^2 + 2^{2k\cdot(s_{\alpha} -1)}2^{2l}\big\|v\big\|_{PW[\kappa]}^2
  \end{align*}
  and we set 
  \begin{align*}
  \big\|v\big\|_{PW[\kappa]}: = \inf_{v = \int_{\omega'}v^{\omega'}}\int_{\omega'\in \kappa}\big\|v^{\omega'}\big\|_{L^2_{t_{\omega'}} L^\infty_{x_{{\omega'}^\perp}}}\,d\omega'
  \end{align*}
  For future reference, we shall use the notation
  \begin{align*}
  \sup_{\omega\notin 2\kappa}\text{dist}(\omega, \kappa)\big\|v\|_{L^\infty_{t_\omega} L^2_{x_{\omega^\perp}}}=: \big\|v\big\|_{NFA[\kappa]^*}
  \end{align*}
  
  For the source terms, we employ the norm associated with the space given by 
  \[
  N_k= L_t^1 \dot{H}^{s_{\alpha - 1}} + \dot{X}_k^{s_\alpha - 1,-\frac12, 1} + NF_k,
  \]
  which involves the somewhat abstract null-frame space $NF_k$ associated to the norm 
  \begin{align*}
  \big\|F\big\|_{NF_k}: = 2^{(s_\alpha - 1)k}\inf_{l>100}\inf_{F = \sum_{\kappa\in K_l}F_{\kappa}^{\pm}}\big(\sum_{\kappa\in K_l}\big\|F_{\kappa}^{\pm}\big\|_{NFA[\pm\kappa]}^2\big)^{\frac12},
  \end{align*}
  where it is understood that in $F = \sum_{\kappa\in K_l}F_{\kappa}^{\pm}$ either the $+$ or the $-$-sign applies everywhere, and each $F^{\pm}_{\kappa}$ has space-time Fourier support contained in $\pm\tau>0, \big||\tau|-|\xi|\big|<2^{k-2l}, \xi\in \kappa$, and we define 
  \begin{align*}
  \big\|F\big\|_{NFA[\kappa]}: = \inf_{\omega'\notin 2\kappa}\text{dist}(\omega', \kappa)^{-1}\big\|F\big\|_{L_{t_{\omega'}}^1 L_{x_{\omega'}^{\perp}}^2}.
  \end{align*}
  Then from \cite{Tao1} we have the key energy inequality for Schwartz functions $\phi_k$ supported at spatial frequency $\sim 2^k$: 
  \begin{equation}\label{eq:energy}
  \big\|\phi_k\big\|_{S_k}\lesssim \big\|\phi_k[0]\big\|_{\dot{H}^{s_{\alpha}}\times \dot{H}^{s_{\alpha - 1}}} + \big\|\Box \phi_k\big\|_{N_k}
  \end{equation}
  The proof of Theorem~\ref{thm:alpharough} is then accomplished by proving the following 
  \begin{prop}\label{prop:alpha rough} Assuming the bounds \eqref{eq:smallu1}, \eqref{eq:smallu2}, as well as the other assumptions of the theorem, and choosing $\epsilon_*$ small enough, the bound
  \begin{align*}
  \sum_{k\in \Z}\big\|P_k v\big\|_{S_k}<\epsilon_*
  \end{align*}
  implies the improved bound 
   \begin{align*}
  \sum_{k\in \Z}\big\|P_k v\big\|_{S_k}<\frac12\epsilon_*
  \end{align*}
  \end{prop}
  
  This proposition, combined with standard arguments (see e. g. \cite{Tao1}, \cite{KrSch}), easily implies the theorem. In turn, in light of \eqref{eq:energy}, the proposition follows from the bound 
  \begin{align*}
  \sum_{k\in \Z}\big\|P_k \Box v\big\|_{N_k}\ll \epsilon_*. 
  \end{align*}
  
  It remains to bound the various terms on the right hand side of \eqref{eq:veqn} with respect to $\sum_k\big\|\cdot\big\|_{N_k}$:
  \\
  
  {\it{(1) Self-interactions of $u_1$.}} Write 
  \begin{align*}
  R_{\nu}^{(\alpha)}u_1 R^{\nu,(\alpha)}u_1 &= \sum_{|k_1 - k_2|<10}R_{\nu}^{(\alpha)}u_{1, k_1} R^{\nu,(\alpha)}u_{1, k_2}\\
  & + \sum_{k_1<k_2-10} R_{\nu}^{(\alpha)}u_{1, k_1} R^{\nu,(\alpha)}u_{1, k_2}\\
  & + \sum_{k_2<k_1-10} R_{\nu}^{(\alpha)}u_{1, k_1} R^{\nu,(\alpha)}u_{1, k_2}\\
  \end{align*}
  The last two terms on the right are similar due to the symmetry. 
  \\
  
{\it{(1.a): Low-high interactions}}. Write this as 
\begin{align*}
\sum_{k_1<k_2-10} R_{\nu}^{(\alpha)}u_{1, k_1} R^{\nu,(\alpha)}u_{1, k_2}& = \sum_{k_1<k_2-10}\sum_{j<k_1}Q_j P_{k_2}[R_{\nu}^{(\alpha)}u_{1, k_1} R^{\nu,(\alpha)}u_{1, k_2}],
\end{align*}
keeping in mind that $u_1$ is a free wave. Observe that we can then write 
\begin{align*}
&Q_j P_{k_2}[R_{\nu}^{(\alpha)}u_{1, k_1} R^{\nu,(\alpha)}u_{1, k_2}]\\
& = \sum_{\pm,\pm,\pm}Q_j^{\pm} P_{k_2}[R_{\nu}^{(\alpha)}u_{1, k_1}^{\pm} R^{\nu,(\alpha)}u_{1, k_2}^{\pm}]\\
& = \sum_{\pm,\pm,\pm}\sum_{\text{dist}(\pm\kappa_1,\pm\kappa_2)\sim 2^{\frac{j-k_1}{2}}}Q_j^{\pm} P_{k_2}[R_{\nu}^{(\alpha)}P_{k_1,\kappa_1}u_{1, k_1}^{\pm} R^{\nu,(\alpha)}P_{k_2,\kappa_2}u_{1, k_2}^{\pm}]\\
\end{align*}

 Then we bound  
\begin{align*}
&\big\|Q_j P_{k_2}[R_{\nu}^{(\alpha)}u_{1, k_1} R^{\nu,(\alpha)}u_{1, k_2}]\big\|_{\dot{X}_{k_2}^{s_{\alpha} - 1,-\frac12, 1}}\\
&\lesssim 2^{-\frac{j}{2}}\cdot 2^{(\frac12 - 2\alpha)k_2}\\&\hspace{1cm}\cdot 2^{j-k_1}\cdot 2^{(1-\alpha)k_1}\cdot 2^{(1-\alpha)k_2}\big(\sum_{\kappa\in K_{\frac{j-k_1}{2}}}\big\|P_{k_1,\kappa}u_1\big\|_{L_t^2 L_x^\infty}^2\big)^{\frac12}\cdot \big(\sum_{\kappa\in K_{\frac{j-k_1}{2}}}\big\|P_{k_2,\kappa}u_1\big\|_{L_t^\infty L_x^2}^2\big)^{\frac12},
\end{align*}
where we have exploited the gain $2^{j-k_1}$ from the null-form due to the angular alignment of the factors. Then if $j>0$ we use the bound 
\begin{align*}
\big(\sum_{\kappa}\big\|P_{k_1,\kappa}u_1\big\|_{L_t^2 L_x^\infty}^2\big)^{\frac12}\lesssim \epsilon_*^{\frac14}\cdot 2^{(2\alpha - 1 + \frac{\alpha}{2+})k_1},
\end{align*}
while we get 
\begin{align*}
\big(\sum_{\kappa}\big\|P_{k_2,\kappa}u_1\big\|_{L_t^\infty L_x^2}^2\big)^{\frac12}\lesssim \epsilon_*\cdot 2^{(2\alpha - \frac32 +\frac{\alpha}{2+})k_2}.
\end{align*}
Inserting these bounds above and simplifying results in 
\begin{align*}
&\big\|Q_j P_{k_2}[R_{\nu}^{(\alpha)}u_{1, k_1} R^{\nu,(\alpha)}u_{1, k_2}]\big\|_{\dot{X}_{k_2}^{s_{\alpha} - 1,-\frac12, 1}}\\
&\lesssim 2^{\frac{j-k_1}{2}}\cdot 2^{\alpha(k_1 - k_2)}\cdot 2^{-\frac{k_1}{2}}\cdot 2^{\frac{\alpha}{2+}k_1 + \frac{\alpha}{2+}k_2}\cdot\epsilon_*^{\frac54}\\
&\lesssim 2^{\frac{j-k_1}{2}}\cdot 2^{\frac{\alpha}{2}(k_1 - k_2)}\cdot 2^{(0-)k_1}  \cdot \epsilon_*^{\frac54},\\
\end{align*}
and one can sum here over $j<k_1, 0<k_1<k_2$. 
 \\
 If $j\leq 0$, then we use the bounds 
 \begin{align*}
 &\big(\sum_{\kappa\in K_{\frac{j-k_1}{2}}}\big\|P_{k_1,\kappa}u_1^{\pm}\big\|_{PW[\pm\kappa]}^2\big)^{\frac12}\lesssim 2^{(2\alpha - \frac12 + \frac{\alpha}{2+})k_1}\cdot 2^{\frac{j-k_1}{2}}\cdot \epsilon_*,\\
 & \big(\sum_{\kappa\in K_{\frac{j-k_1}{2}}}\big\|P_{k_1,\kappa}u_1^{\pm}\big\|_{NFA[\pm\kappa]^*}^2\big)^{\frac12}\lesssim 2^{(2\alpha - \frac32 + \frac{\alpha}{2+})k_1}\cdot \epsilon_*,\\
 \end{align*}
 which gives 
 \begin{align*}
 &\big\|Q_j P_{k_2}[R_{\nu}^{(\alpha)}u_{1, k_1} R^{\nu,(\alpha)}u_{1, k_2}]\big\|_{\dot{X}_{k_2}^{s_{\alpha} - 1,-\frac12, 1}}\\
 &\lesssim 2^{\frac{j}{2}}\cdot 2^{\alpha(k_1 - k_2)}\cdot 2^{-\frac{k_1}{2}}\cdot 2^{\frac{\alpha}{2+}k_1 + \frac{\alpha}{2+}k_2}\cdot\epsilon_*^2\\
 &\lesssim 2^{\frac{j}{2}}\cdot 2^{\frac{\alpha}{2}(k_1 - k_2)}\cdot 2^{(0-)k_1}\cdot \epsilon_*^2,
 \end{align*}
 which can be summed over $j<0, 0<k_1<k_2$. 
 \\
 
 {\it{(1.b): High-high interactions}}. Write this as 
 \begin{align*}
 &\sum_{|k_1-k_2|<10}\sum_{k<k_1 +10}\sum_{j<k+10}Q_j P_{k}[R_{\nu}^{(\alpha)}u_{1, k_1} R^{\nu,(\alpha)}u_{1, k_2}]\\
 & + \sum_{|k_1-k_2|<10}\sum_{k<k_1 +10}\sum_{j\geq k+10}Q_j P_{k}[R_{\nu}^{(\alpha)}u_{1, k_1} R^{\nu,(\alpha)}u_{1, k_2}]\\
 \end{align*}
 Then we bound the first term on the right(with $j<k+10$) by 
 \begin{align*}
 &\big\|Q_j P_{k}[R_{\nu}^{(\alpha)}u_{1, k_1} R^{\nu,(\alpha)}u_{1, k_2}]\big\|_{\dot{X}_{k}^{s_{\alpha} - 1,-\frac12, 1}}\\
 &\lesssim 2^{-\frac{j}{2}}\cdot 2^{(\frac12 - 2\alpha)k}\cdot 2^{2(k-k_1) + j-k}\\&\hspace{1cm}\cdot 2^{(1-\alpha)k_1}\big(\sum_{\kappa}\|P_{k_1,\kappa}u_1\|_{L_t^2 L_x^\infty}^2\big)^{\frac12}\cdot 2^{(1-\alpha)k_2}\big(\sum_{\kappa}\|P_{k_2,\kappa}u_{1}\|_{ L_t^\infty L_x^2}^2\big)^{\frac12},\\
 \end{align*}
 where it is understood that the $\kappa$ range over $K_l$ with $l = k-k_1 + \frac{j-k}{2}$. 
 Using the bounds from before for the square sums over caps, we find provided $\frac{j+k}{2}\geq 0$
 \begin{align*}
 &\big\|Q_j P_{k}[R_{\nu}^{(\alpha)}u_{1, k_1} R^{\nu,(\alpha)}u_{1, k_2}]\big\|_{\dot{X}_{k}^{s_{\alpha} - 1,-\frac12, 1}}\\
 &\lesssim  2^{-\frac{j}{2}}\cdot 2^{(\frac12 - 2\alpha)k}\cdot 2^{2(k-k_1) + j-k}\cdot  2^{(1-\alpha)k_1}\cdot  2^{(1-\alpha)k_2}\\
 &\hspace{2cm}\cdot 2^{(2\alpha - 1 + \frac{\alpha}{2+})k_1}\cdot 2^{(2\alpha - \frac32 +\frac{\alpha}{2+})k_2}\cdot \epsilon_*^{\frac54}\\
 &\lesssim 2^{\frac{j-k}{2}}\cdot 2^{2(1-\alpha)(k-k_1)}\cdot 2^{(-\frac12 + 2\cdot \frac{\alpha}{2+})k_2}\cdot \epsilon_*^{\frac54},\\
 \end{align*}
 One can sum here over $j<k$ as well as $k<k_1 = k_2+O(1)$. The case $\frac{j+k}{2}<0$ is again handled by using the $PW[\kappa], NFA[\kappa]^*$ norms, analogously to the preceding case. 
 \\
 As for the term with large modulation, as the factors are free waves, we have (when $j\geq k+10$)
 \[
 Q_j P_{k}[R_{\nu}^{(\alpha)}u_{1, k_1} R^{\nu,(\alpha)}u_{1, k_2}] = Q_{k_1+O(1)} P_{k}[R_{\nu}^{(\alpha)}u_{1, k_1} R^{\nu,(\alpha)}u_{1, k_2}], 
 \]
 and it follows that we then have 
 \begin{align*}
  &\big\|Q_j P_{k}[R_{\nu}^{(\alpha)}u_{1, k_1} R^{\nu,(\alpha)}u_{1, k_2}]\big\|_{\dot{X}_{k}^{s_{\alpha} - 1,-\frac12, 1}}\\
  &\lesssim 2^{-\frac{k_1}{2}}\cdot 2^{(\frac12 - 2\alpha)k}\cdot 2^{(1-\alpha)k_1}\cdot  2^{(1-\alpha)k_2}\\
 &\hspace{2cm}\cdot 2^{(2\alpha - 1 + \frac{\alpha}{2+})k_1}\cdot 2^{(2\alpha - \frac32 +\frac{\alpha}{2+})k_2}\cdot \epsilon_*^{\frac54}\\
 &\lesssim 2^{(\frac12 - 2\alpha)k}\cdot 2^{(2\alpha-1 + 2\frac{\alpha}{2+})k_1}\cdot \epsilon_*^{\frac54}\\
 &\leq 2^{(\frac12 - 2\alpha)k}\cdot 2^{(2\alpha-\frac12)k_1}\cdot 2^{(0-)k_1}\cdot \epsilon_*^{\frac54}\\
 \end{align*}
 This is summable, recalling $\alpha<\frac14$. 
 \\
This concludes the estimates for the self-interactions of $u_1$, i. e. the first term on the right hand side of \eqref{eq:veqn}. 
 \\
 
 {\it{(2): Mixed interactions between $u_1$ and $v$.}} Write 
 \begin{align*}
  R_{\nu}^{(\alpha)}u_1 R^{\nu,(\alpha)}v & =  \sum_{|k_1 - k_2|<10}R_{\nu}^{(\alpha)}u_{1,k_1} R^{\nu,(\alpha)}v_{k_2}\\
  & + \sum_{k_1<k_2-10}R_{\nu}^{(\alpha)}u_{1,k_1} R^{\nu,(\alpha)}v_{k_2}\\
  & +  \sum_{k_2<k_1-10}R_{\nu}^{(\alpha)}u_{1,k_1} R^{\nu,(\alpha)}v_{k_2},\\
 \end{align*}
 where this time the last two terms are no longer identical. 
 \\
 
 {\it{(2.a): low-high interactions, i. e. the second term on the right.}} We decompose it further into a number of terms: 
 \begin{align*}
 R_{\nu}^{(\alpha)}u_{1,k_1} R^{\nu,(\alpha)}v_{k_2} & = P_{k_2}Q_{>k_1+10}[R_{\nu}^{(\alpha)}u_{1,k_1} R^{\nu,(\alpha)}v_{k_2}]\\
 & + \sum_{j<k_1+10}P_{k_2}Q_{j}[R_{\nu}^{(\alpha)}u_{1,k_1} R^{\nu,(\alpha)}Q_{<j}v_{k_2}]\\
 & + \sum_{j<k_1+10}P_{k_2}Q_{<j}[R_{\nu}^{(\alpha)}u_{1,k_1} R^{\nu,(\alpha)}Q_jv_{k_2}]\\
 \end{align*}
 We bound each of the terms on the right in turn, the first being easier: 
 \begin{align*}
&\|P_{k_2}Q_{>k_1+10}[R_{\nu}^{(\alpha)}u_{1,k_1} R^{\nu,(\alpha)}v_{k_2}]\|_{\dot{X}_{k}^{s_{\alpha} - 1,-\frac12, 1}}\\
&\lesssim 2^{-\frac{k_1}{2}}2^{(\frac12 - 2\alpha)k_2}\cdot 2^{(1-\alpha)k_1}\big\|P_{k_1}u_1\big\|_{L_t^2 L_x^\infty}\cdot 2^{(1-\alpha)k_2}\big(2^{-k_2}\big\|\nabla_{t,x}v_{k_2}\big\|_{L_t^\infty L_x^2}\big)\\
&\lesssim 2^{-\frac{k_1}{2}}2^{(\frac12 - 2\alpha)k_2}\cdot 2^{(1-\alpha)k_1}\cdot 2^{(2\alpha - 1)k_1}\cdot 2^{\frac{\alpha}{2+}k_1}\epsilon_*^{\frac14}\\
&\hspace{4cm}\cdot 2^{(1-\alpha)k_2}\cdot 2^{(2\alpha - \frac32)k_2}\cdot\epsilon_*\\
&\lesssim 2^{(\alpha - \frac12 + \frac{\alpha}{2+})k_1}\cdot 2^{-\alpha k_2}\cdot\epsilon_*^{\frac54},
 \end{align*}
 which is summable over $0<k_1<k_2$ for $0<\alpha<\frac14$, say. 
 \\
 For the second and third term on the right above, we have to take advantage of the null-structure: 
 \begin{align*}
2P_{k_2}Q_{j}[R_{\nu}^{(\alpha)}u_{1,k_1} R^{\nu,(\alpha)}Q_{<j}v_{k_2}]& = P_{k_2}Q_{j}\Box [|\nabla|^{-\alpha}u_{1,k_1} Q_{<j}|\nabla|^{-\alpha}v_{k_2}]\\
& - P_{k_2}Q_{j}[|\nabla|^{-\alpha}u_{1,k_1} Q_{<j}|\nabla|^{-\alpha}\Box v_{k_2}]\\
 \end{align*}
 Then we bound 
 \begin{align*}
 &\big\|P_{k_2}Q_{j}\Box [|\nabla|^{-\alpha}u_{1,k_1} Q_{<j}|\nabla|^{-\alpha}v_{k_2}]\big\|_{\dot{X}_{k}^{s_{\alpha} - 1,-\frac12, 1}}\\
 &\lesssim 2^{\frac{j}{2}}\cdot 2^{(\frac32 - 2\alpha)k_2}\cdot 2^{-\alpha k_1}\cdot\big\|u_{1,k_1}\big\|_{L_t^2 L_x^\infty}\cdot 2^{-\alpha k_2}\big\|v_{k_2}\big\|_{L_t^\infty L_x^2}\\
 &\lesssim 2^{\frac{j}{2}}\cdot 2^{(\frac32 - 2\alpha)k_2}\cdot 2^{-\alpha k_1}\cdot  2^{(2\alpha - 1)k_1}\cdot 2^{\frac{\alpha}{2+}k_1}\epsilon_*^{\frac14}\\
 &\hspace{4cm}\cdot 2^{-\alpha k_2}2^{(2\alpha - \frac32)k_2}\cdot \epsilon_*\\
 &\lesssim 2^{\frac{j-k_1}{2}}\cdot 2^{(\alpha + \frac{\alpha}{2+} - \frac12)k_1}\cdot 2^{-\alpha k_2}\epsilon_*^{\frac54},
 \end{align*}
 and this can be summed over $j<k_1+O(1)$, $0<k_1<k_2$, provided $0<\alpha<\frac14$. 
  \\
  
  Further, we have 
  \begin{align*}
  &\big\| P_{k_2}Q_{j}[|\nabla|^{-\alpha}u_{1,k_1} Q_{<j}|\nabla|^{-\alpha}\Box v_{k_2}]\big\|_{L_t^1 \dot{H}^{s_{\alpha} - 1}}\\
  &\lesssim 2^{(\frac12 - 2\alpha)k_2}\cdot 2^{-\alpha k_1}\big\|u_{1,k_1}\big\|_{L_t^2 L_x^\infty}\cdot 2^{-\alpha k_2}\big\|\Box Q_{<j}v_{k_2}\big\|_{L_{t,x}^2}\\
  &\lesssim 2^{(\frac12 - 2\alpha)k_2}\cdot 2^{-\alpha k_1}\cdot 2^{(2\alpha - 1 + \frac{\alpha}{2+})k_1}\cdot 2^{-\alpha k_2}\cdot 2^{\frac{j}{2}}\cdot 2^{(2\alpha - \frac12)k_2}\epsilon_*^{\frac54},
  \end{align*}
  which leads to the same bound as in the preceding case. We note that it is here that the larger range of Strichartz estimates appears crucial.
  The term with $Q_j, Q_{<j}$ interchanged is handled in the same way. 
  \\
  
  {\it{(2.b): high - low interactions, i. e. the expression $\sum_{k_2<k_1-10}R_{\nu}^{(\alpha)}u_{1,k_1} R^{\nu,(\alpha)}v_{k_2}$.}} Decompose this for fixed $k_2<k_1 - 10$ into 
  \begin{align*}
  R_{\nu}^{(\alpha)}u_{1,k_1} R^{\nu,(\alpha)}v_{k_2}& = P_{k_1}Q_{>k_2+10}[R_{\nu}^{(\alpha)}u_{1,k_1} R^{\nu,(\alpha)}v_{k_2}]\\
  & + \sum_{j\leq k_2+10}P_{k_1}Q_{j}[R_{\nu}^{(\alpha)}u_{1,k_1} R^{\nu,(\alpha)}Q_{<j}v_{k_2}]\\
  & +  \sum_{j\leq k_2+10}P_{k_1}Q_{<j}[R_{\nu}^{(\alpha)}u_{1,k_1} R^{\nu,(\alpha)}Q_{j}v_{k_2}]\\
  \end{align*}
  Observe that for the first term on the right, we have 
  \begin{align*}
  P_{k_1}Q_{>k_2+10}[R_{\nu}^{(\alpha)}u_{1,k_1} R^{\nu,(\alpha)}v_{k_2}] = P_{k_1}Q_{>k_2+10}[R_{\nu}^{(\alpha)}u_{1,k_1} R^{\nu,(\alpha)}Q_{>k_2}v_{k_2}]
  \end{align*}
  Then we get 
  \begin{align*}
  &\big\|P_{k_1}Q_{>k_2+10}[R_{\nu}^{(\alpha)}u_{1,k_1} R^{\nu,(\alpha)}Q_{>k_2}v_{k_2}]\big\|_{\dot{X}_{k}^{s_{\alpha} - 1,-\frac12, 1}}\\
&\lesssim 2^{-\frac{k_2}{2}}\cdot 2^{(\frac12 - 2\alpha)k_1}\cdot 2^{(1-\alpha)k_1}\cdot \|u_{k_1}\|_{L_{t,x}^\infty}\cdot 2^{(1-\alpha)k_2}\cdot \big(2^{-k_2}\|\nabla_{t,x}Q_{>k_2}v_{k_2}\|_{L_{t,x}^2}\big)\\
&\lesssim 2^{-\frac{k_2}{2}}\cdot 2^{(\frac12 - 2\alpha)k_1}\cdot 2^{(1-\alpha)k_1}\cdot 2^{(2\alpha - \frac32 + \frac{\alpha}{2+})k_1}\cdot 2^{(1-\alpha)k_2}\cdot 2^{-\frac{k_2}{2}}\cdot 2^{(2\alpha - \frac32)k_2}\epsilon_*^{\frac54}\\
&\lesssim 2^{-\frac{\alpha}{2} k_1}\cdot 2^{(\alpha -\frac32)k_2}\epsilon_*^{\frac54}, 
  \end{align*}
which can be summed over $0<k_2<k_1-10$, if $0<\alpha<1$. If we restrict $k_2\leq 0$, we instead use 
\begin{align*}
 &\big\|P_{k_1}Q_{>k_2+10}[R_{\nu}^{(\alpha)}u_{1,k_1} R^{\nu,(\alpha)}Q_{>k_2}v_{k_2}]\big\|_{\dot{X}_{k}^{s_{\alpha} - 1,-\frac12, 1}}\\
 &\lesssim  2^{-\frac{k_2}{2}}\cdot 2^{(\frac12 - 2\alpha)k_1}\cdot 2^{(1-\alpha)k_1}\cdot \|u_{k_1}\|_{L_t^\infty L_x^2}\cdot 2^{(1-\alpha)k_2}\cdot \big(2^{-k_2}\|\nabla_{t,x}Q_{>k_2}v_{k_2}\|_{L_t^2 L_x^\infty}\big)\\
 &\lesssim  2^{-\frac{k_2}{2}}\cdot 2^{(\frac12 - 2\alpha)k_1}\cdot 2^{(1-\alpha)k_1}\cdot 2^{(2\alpha - \frac32 + \frac{\alpha}{2+})k_1}\cdot 2^{(1-\alpha)k_2}\cdot 2^{(2\alpha - \frac12)k_2}\epsilon_*^{\frac54}\\
 &\lesssim 2^{-\frac{\alpha}{2}k_1}\cdot 2^{\alpha k_2}\cdot \epsilon_*^{\frac54},\\
\end{align*}
which can be summed over $k_2\leq 0, k_1>0$. 
 \\
 
 Next, we again use the null-structure to write 
 \begin{align*}
 2P_{k_1}Q_{j}[R_{\nu}^{(\alpha)}u_{1,k_1} R^{\nu,(\alpha)}Q_{<j}v_{k_2}]& = \Box P_{k_1}Q_{j}[|\nabla|^{-\alpha}u_{1,k_1} |\nabla|^{-\alpha}Q_{<j}v_{k_2}]\\
 & - P_{k_1}Q_{j}[|\nabla|^{-\alpha}u_{1,k_1} |\nabla|^{-\alpha}Q_{<j}\Box v_{k_2}]
 \end{align*}
  Then we bound 
  \begin{align*}
  &\big\|\Box P_{k_1}Q_{j}[|\nabla|^{-\alpha}u_{1,k_1} |\nabla|^{-\alpha}Q_{<j}v_{k_2}]\big\|_{\dot{X}_{k}^{s_{\alpha} - 1,-\frac12, 1}}\\
  &\lesssim 2^{\frac{j}{2}}\cdot 2^{(\frac32 - 2\alpha)k_1}\cdot 2^{-\alpha k_1}\cdot \big\|u_{k_1}\big\|_{L_t^p L_x^{\frac{2p}{p-1}+}}\cdot 2^{-\alpha k_2}\big\| v_{k_2}\big\|_{L_t^r L_x^{2p-}}\\
 \end{align*}
  where we have $\frac{1}{p} + \frac{1}{r} = \frac12$, $p$ large. Observe that $(r, 2p-)$ is (standard) Strichartz admissible in $3+1$ dimensions.   Then by interpolating between $L_t^2 L_x^{4+}$ and $L_t^\infty L_x^2$, recalling \eqref{eq:smallu22}, and picking $p$ very large, we bound the preceding by 
\begin{align*}
&2^{\frac{j}{2}}\cdot 2^{(\frac32 - 2\alpha)k_1}\cdot 2^{-\alpha k_1}\cdot \big\|u_{k_1}\big\|_{L_t^p L_x^{\frac{2p}{p-1}+}}\cdot 2^{-\alpha k_2}\big\| v_{k_2}\big\|_{L_t^r L_x^{2p-}}\\
&\lesssim 2^{\frac{j}{2}}\cdot 2^{(\frac32 - 2\alpha)k_1}\cdot 2^{-\alpha k_1}\cdot 2^{(2\alpha - \frac32+\frac{\alpha}{2+})k_1}\cdot 2^{-\alpha k_2}\cdot 2^{(2\alpha - \frac12)k_2}\cdot\epsilon_*^{\frac54}\\
\end{align*}
This can be summed over $k_2<\frac12 k_1>0$, $j<k_2$. Hence assume now $k_2\geq \frac12 k_1$. Then we use 
 \begin{align*}
  &\big\|\Box P_{k_1}Q_{j}[|\nabla|^{-\alpha}u_{1,k_1} |\nabla|^{-\alpha}Q_{<j}v_{k_2}]\big\|_{\dot{X}_{k}^{s_{\alpha} - 1,-\frac12, 1}}\\
  &\lesssim 2^{\frac{j}{2}}\cdot 2^{(\frac32 - 2\alpha)k_1}\cdot 2^{-\alpha k_1}\cdot \big\| u_{k_1}\big\|_{L_t^2 L_x^\infty}\cdot 2^{-\alpha k_2}\cdot \big\|v_{k_2}\big\|_{L_t^\infty L_x^2}\\
  &\lesssim 2^{\frac{j-k_2}{2}}\cdot 2^{(\frac32 - 2\alpha)k_1}\cdot 2^{-\alpha k_1}\cdot 2^{(2\alpha - 1 + \frac{\alpha}{2+})k_1}\cdot 2^{-\alpha k_2}\cdot 2^{(2\alpha - 1)k_2}\cdot\epsilon_*^{\frac54}\\
  &\lesssim  2^{\frac{j-k_2}{2}}\cdot 2^{(\frac12 - \alpha + \frac{\alpha}{2+})k_1}\cdot 2^{(\frac{\alpha}{2} - \frac12)k_1}\cdot\epsilon_*^{\frac54},\\
  \end{align*}
  which can be summed over $0<k_2<k_1<2k_2,\,j<k_2$. 
  \\
  
  Next, consider the term $P_{k_1}Q_{j}[|\nabla|^{-\alpha}u_{1,k_1} |\nabla|^{-\alpha}Q_{<j}\Box v_{k_2}]$. We bound this by 
  \begin{align*}
  &\big\|P_{k_1}Q_{j}[|\nabla|^{-\alpha}u_{1,k_1} |\nabla|^{-\alpha}Q_{<j}\Box v_{k_2}]\big\|_{L_t^1 \dot{H}^{s_\alpha - 1}}\\
  &\lesssim 2^{(\frac12 - 2\alpha)k_1}\cdot 2^{-\alpha k_1}\big\|u_{1,k_1}\big\|_{L_t^2 L_x^\infty}\cdot 2^{-\alpha k_2}\big\|Q_{<j}\Box v_{k_2}\big\|_{L_{t,x}^2}\\
  &\lesssim 2^{(\frac12 - 2\alpha)k_1}\cdot 2^{-\alpha k_1}\cdot 2^{(2\alpha - 1+ \frac{\alpha}{2+})k_1}\cdot 2^{-\alpha k_2}\cdot 2^{\frac{j}{2}}\cdot 2^{(2\alpha - \frac12)k_2}\cdot \epsilon_*^{\frac54},
  \end{align*}
  and this can be summed over $k_2<k_1$, $j<k_2$.
  \\
  
  The remaining term $\sum_{j\leq k_2+10}P_{k_1}Q_{<j}[R_{\nu}^{(\alpha)}u_{1,k_1} R^{\nu,(\alpha)}Q_{j}v_{k_2}]$ is handled similarly. 
  \\
  
  {\it{(2.c): High-high interactions}}. This is the expression 
  \[
  \sum_{|k_1 - k_2|<10}R_{\nu}^{(\alpha)}u_{1,k_1} R^{\nu,(\alpha)}v_{k_2} = \sum_{|k_1 - k_2|<10}\sum_{k<k_1+10}P_k\big[R_{\nu}^{(\alpha)}u_{1,k_1} R^{\nu,(\alpha)}v_{k_2}\big].
  \]
  Fixing $k, k_{1,2}$, we decompose the term further into 
  \begin{equation}\label{eq:hhmess}\begin{split}
  P_k\big[R_{\nu}^{(\alpha)}u_{1,k_1} R^{\nu,(\alpha)}v_{k_2}\big]&= P_k\big[R_{\nu}^{(\alpha)}u_{1,k_1} R^{\nu,(\alpha)}Q_{\geq k}v_{k_2}\big]\\
  & + P_kQ_{\geq k+10}\big[R_{\nu}^{(\alpha)}u_{1,k_1} R^{\nu,(\alpha)}Q_{<k}v_{k_2}\big]\\
  & + \sum_{j<k+10}P_kQ_{j}\big[R_{\nu}^{(\alpha)}u_{1,k_1} R^{\nu,(\alpha)}Q_{<j}v_{k_2}\big]\\
  & +  \sum_{j<k+10}P_kQ_{<j}\big[R_{\nu}^{(\alpha)}u_{1,k_1} R^{\nu,(\alpha)}Q_{j}v_{k_2}\big]\\
  \end{split}\end{equation}
 We estimate each of these terms in turn.  For the first term on the right, write it as 
 \begin{align*}
 &P_k\big[R_{\nu}^{(\alpha)}u_{1,k_1} R^{\nu,(\alpha)}Q_{\geq k}v_{k_2}\big]\\
 & = P_k\big[R_{\nu}^{(\alpha)}u_{1,k_1} R^{\nu,(\alpha)}Q_{k_1-10>\cdot\geq k}v_{k_2}\big]\\
 & + P_k\big[R_{\nu}^{(\alpha)}u_{1,k_1} R^{\nu,(\alpha)}Q_{\geq k_1-10}v_{k_2}\big].\\
 \end{align*}
  Then bound the second term on the right by 
  \begin{align*}
  &\big\|P_k\big[R_{\nu}^{(\alpha)}u_{1,k_1} R^{\nu,(\alpha)}Q_{\geq k_1-10}v_{k_2}\big]\big\|_{L_t^1 \dot{H}^{s_{\alpha} - 1}}\\
 &\lesssim 2^{(\frac12 - 2\alpha)k}\cdot 2^{(1-\alpha)k_1}\cdot \big\|u_{1,k_1}\big\|_{L_t^2 L_x^\infty}\cdot 2^{(1-\alpha)k_2}\big(2^{-k_2}\big\|\nabla_{t,x}Q_{\geq k_1-10}v_{k_2}\big\|_{L_{t,x}^2}\big)\\
 &\lesssim 2^{(\frac12 - 2\alpha)k}\cdot 2^{(1-\alpha)k_1}\cdot 2^{(2\alpha - 1 + \frac{\alpha}{2+})k_1}\cdot 2^{(1-\alpha)k_2}\cdot 2^{-\frac{k_1}{2}}\cdot 2^{(2\alpha - \frac32)k_2}\cdot\epsilon_*^{\frac54}\\
 &\lesssim 2^{(\frac12 - 2\alpha)k}\cdot 2^{(2\alpha - 1 + \frac{\alpha}{2+})k_1}\cdot \epsilon_*^{\frac54},\\
  \end{align*}
  which can be summed over $k<k_1 = k_2+O(1)$, $k_1>0$, provided $\alpha <\frac14$. Next, consider the first term on the right above, which is a bit more subtle. In fact, we can decompose it further into
  \begin{align*}
  P_k\big[R_{\nu}^{(\alpha)}u_{1,k_1} R^{\nu,(\alpha)}Q_{k_1-10>\cdot\geq k}v_{k_2}\big]& = \sum_{\pm}P_k\big[R_{\nu}^{(\alpha)}u_{1,k_1}^{\pm} R^{\nu,(\alpha)}Q_{k_1-10>\cdot\geq k}^{\pm}v_{k_2}\big]\\
  & +  \sum_{\pm}P_k\big[R_{\nu}^{(\alpha)}u_{1,k_1}^{\pm} R^{\nu,(\alpha)}Q_{k_1-10>\cdot\geq k}^{\mp}v_{k_2}\big]\\
  \end{align*}
  Then the fact that $u_{1,k_1}$ is a free wave implies 
  \begin{align*}
  &\sum_{\pm}P_k\big[R_{\nu}^{(\alpha)}u_{1,k_1}^{\pm} R^{\nu,(\alpha)}Q_{k_1-10>\cdot\geq k}^{\pm}v_{k_2}\big]\\& = \sum_{\pm}P_kQ_{\geq k_1-10}\big[R_{\nu}^{(\alpha)}u_{1,k_1}^{\pm} R^{\nu,(\alpha)}Q_{k_1-10>\cdot\geq k}^{\pm}v_{k_2}\big],\\
  \end{align*}
  and so we can bound it by 
  \begin{align*}
  &\big\|\sum_{\pm}P_k\big[R_{\nu}^{(\alpha)}u_{1,k_1}^{\pm} R^{\nu,(\alpha)}Q_{k_1-10>\cdot\geq k}^{\pm}v_{k_2}\big]\big\|_{\dot{X}^{s_{\alpha} - 1,-\frac12, 1}}\\
  &\lesssim  2^{(\frac12 - 2\alpha)k}\cdot 2^{-\frac{k_1}{2}}\cdot 2^{(1-\alpha)k_1}\cdot\big\|u_{1,k_1}^{\pm}\big\|_{L_{t,x}^\infty}\cdot 2^{(1-\alpha)k_2}\cdot \big\|Q_{k_1-10>\cdot\geq k}^{\pm}v_{k_2}\big\|_{L_{t,x}^2}\\
  &\lesssim  2^{(\frac12 - 2\alpha)k}\cdot 2^{-\frac{k_1}{2}}\cdot 2^{(1-\alpha)k_1}\cdot 2^{(2\alpha - \frac32 + \frac{\alpha}{2+})k_1}\cdot 2^{(1-\alpha)k_2}\cdot 2^{-\frac{k}{2}}\cdot 2^{(2\alpha - \frac32)k_2}\cdot\epsilon_*^{\frac54}\\
  &\lesssim 2^{-2\alpha k}\cdot 2^{(2\alpha - \frac32 +  \frac{\alpha}{2+})k_1}\cdot \epsilon_*^{\frac54},\\
  \end{align*}
  which can be summed over $0<k<k_1 = k_2+O(1)$ provided $0<\alpha<\frac14$, say. To handle the case $k<0$, one changes $L_{t,x}^\infty$ to $L_t^\infty L_x^2$ and applies Bernstein's inequality to the whole expression to place it into $L_{t,x}^2$. 
  \\
  
  Consider now the term $\sum_{\pm}P_k\big[R_{\nu}^{(\alpha)}u_{1,k_1}^{\pm} R^{\nu,(\alpha)}Q_{k_1-10>\cdot\geq k}^{\mp}v_{k_2}\big]$. Here the presence of the two derivatives $\partial_{\nu}, \partial^{\nu}$ gains a factor $2^{l-k_1}$ if we fix the modulation of the term $Q_{k_1-10>\cdot\geq k}^{\mp}v_{k_2}$ to size $\sim 2^l$, and so we can bound this by 
 \begin{align*}
 &\big\|\sum_{\pm}P_k\big[R_{\nu}^{(\alpha)}u_{1,k_1}^{\pm} R^{\nu,(\alpha)}Q_{k_1-10>\cdot\geq k}^{\mp}v_{k_2}\big]\big\|_{L_t^1 \dot{H}^{s_{\alpha} - 1}}\\
  &\lesssim \sum_{l\in [k, k_1 - 10]}2^{(\frac12 - 2\alpha)k}\cdot 2^{(1-\alpha)k_1}\cdot 2^{l-k_1}\cdot\big\|u_{1,k_1}^{\pm}\big\|_{L_t^2 L_x^\infty}\cdot 2^{(1-\alpha)k_2}\cdot\big\|Q_{l}^{\mp}v_{k_2}\big\|_{L_{t,x}^2}\\
  &\lesssim \sum_{l\in [k, k_1 - 10]}2^{(\frac12 - 2\alpha)k}\cdot 2^{(1-\alpha)k_1}\cdot 2^{l - k_1}\cdot 2^{(2\alpha - 1 + \frac{\alpha}{2+})k_1}\cdot 2^{(1-\alpha)k_2}\cdot 2^{-\frac{l}{2}}\cdot 2^{(2\alpha - \frac32)k_2}\epsilon_*^{\frac54}\\
  &\lesssim \sum_{l\in [k, k_1 - 10]} 2^{(\frac12-2\alpha) k}\cdot 2^{\frac{l-k_1}{2}}\cdot 2^{(2\alpha - 1 + \frac{\alpha}{2+})k_1}\cdot\epsilon_*^{\frac54},\\
 \end{align*}
  and this can be summed over $k<k_1 =k_2+O(1)>0$. 
  \\
  
  The second term on the right of \eqref{eq:hhmess} is treated by observing that 
  \[
  P_kQ_{\geq k+10}\big[R_{\nu}^{(\alpha)}u_{1,k_1} R^{\nu,(\alpha)}Q_{<k}v_{k_2}\big] = P_kQ_{>k_1-10}\big[R_{\nu}^{(\alpha)}u_{1,k_1} R^{\nu,(\alpha)}Q_{<k}v_{k_2}\big],
  \]
  due to the fact that $u_{1,k_1}$ is a free wave, and this can then be bounded by 
  \begin{align*}
  &\big\|P_kQ_{>k_1-10}\big[R_{\nu}^{(\alpha)}u_{1,k_1} R^{\nu,(\alpha)}Q_{<k}v_{k_2}\big]\big\|_{\dot{X}^{s_{\alpha} - 1,-\frac12, 1}}\\
  &\lesssim 2^{-\frac{k_1}{2}}\cdot 2^{(\frac12 - 2\alpha)k}\cdot 2^{(1-\alpha)k_1}\cdot \big\|u_{1,k_1}\big\|_{L_t^2 L_x^\infty}\cdot 2^{(1-\alpha)k_2}\big\|Q_{<k}v_{k_2}\big\|_{L_t^\infty L_x^2}\\
  &\lesssim 2^{-\frac{k_1}{2}}\cdot 2^{(\frac12 - 2\alpha)k}\cdot 2^{(1-\alpha)k_1}\cdot 2^{(2\alpha - 1 + \frac{\alpha}{2+})k_1}\cdot 2^{(1-\alpha)k_2}\cdot 2^{(2\alpha - \frac32)k_2}\cdot\epsilon_*^{\frac54}\\
  &\lesssim 2^{(\frac12 - 2\alpha)k}\cdot 2^{(2\alpha - 1 + \frac{\alpha}{2+})k_1}\cdot\epsilon_*^{\frac54}.\\
  \end{align*}
  which can be summed over $k<k_1 = k_2+O(1)>0$ provided $\alpha<\frac14$. 
  \\
  
  As for the third and fourth terms in \eqref{eq:hhmess}, they are handled similarly, and so we consider only the fourth term, which we expand as usual: 
 \begin{align*}
 2P_kQ_{<j}\big[R_{\nu}^{(\alpha)}u_{1,k_1} R^{\nu,(\alpha)}Q_{j}v_{k_2}\big]& = \Box P_kQ_{<j}\big[|\nabla|^{-\alpha}u_{1,k_1} |\nabla|^{-\alpha}Q_{j}v_{k_2}\big]\\
 & - P_kQ_{<j}\big[|\nabla|^{-\alpha}u_{1,k_1} |\nabla|^{-\alpha}Q_{j}\Box v_{k_2}\big]\\
 \end{align*}
 Then we get 
 \begin{align*}
&\big\|\Box P_kQ_{<j}\big[|\nabla|^{-\alpha}u_{1,k_1} |\nabla|^{-\alpha}Q_{j}v_{k_2}\big]\big\|_{\dot{X}^{s_{\alpha} - 1,-\frac12, 1}}\\
&\lesssim 2^{\frac{j}{2}} 2^{(\frac32 - 2\alpha)k}\cdot 2^{-\alpha k_1}\big\|u_{1,k_1}\big\|_{L_t^2 L_x^\infty}\cdot 2^{-\alpha k_2}\cdot \big\|Q_{j}v_{k_2}\big\|_{L_t^\infty L_x^2}\\
&\lesssim  2^{\frac{j}{2}}\cdot  2^{(\frac32 - 2\alpha)k}\cdot 2^{-\alpha k_1}\cdot 2^{(2\alpha - 1 + \frac{\alpha}{2+})k_1}\cdot 2^{-\alpha k_2}\cdot 2^{(2\alpha - \frac32)k_2}\cdot \epsilon_*^{\frac54}\\
&\lesssim 2^{\frac{j-k}{2}}\cdot 2^{(2 - 2\alpha)k}\cdot 2^{(2\alpha - \frac52 + \frac{\alpha}{2+})k_1}\cdot \epsilon_*^{\frac54},\\
 \end{align*}
 which is summable over $j<k, k<k_1 = k_2 + O(1)$ for $\alpha<\frac14$.  Further, we have 
 \begin{align*}
 &\big\|P_kQ_{<j}\big[|\nabla|^{-\alpha}u_{1,k_1} |\nabla|^{-\alpha}Q_{j}\Box v_{k_2}\big]\big\|_{L_t^1 \dot{H}^{s_{\alpha} - 1}}\\
 &\lesssim 2^{(\frac12 - 2\alpha)k}\cdot 2^{-\alpha k_1}\cdot \big\|u_{1,k_1}\big\|_{L_t^2 L_x^\infty}\cdot 2^{-\alpha k_2}\big\|Q_{j}\Box v_{k_2}\big\|_{L_{t,x}^2}\\
 &\lesssim  2^{(\frac12 - 2\alpha)k}\cdot 2^{-\alpha k_1}\cdot 2^{(2\alpha - 1 +  \frac{\alpha}{2+})k_1}\cdot 2^{-\alpha k_2}\cdot 2^{\frac{j}{2}}\cdot 2^{(2\alpha - \frac12)k_2}\cdot\epsilon_*^{\frac54}\\
 &\lesssim 2^{\frac{j-k}{2}}\cdot 2^{(1-2\alpha)k}\cdot 2^{(2\alpha - \frac32 + \frac{\alpha}{2+})k_1}\cdot\epsilon_*^{\frac54},\\
 \end{align*}
 which can be summed over $j<k, k<k_1 = k_2+O(1)>0$, proved $\alpha<\frac14$. 
 \\
 
 {\it{(3): self-interactions of $v$.}} Here we bound the term $R_{\nu}^{(\alpha)}v R^{\nu,(\alpha)}v$, which can be achieved by means of the now well-known null-frame type spaces of Tataru. Decompose as usual 
 \begin{align*}
 R_{\nu}^{(\alpha)}v R^{\nu,(\alpha)}v& = \sum_{|k_1 - k_2|<10}R_{\nu}^{(\alpha)}v_{k_1} R^{\nu,(\alpha)}v_{k_2}\\
 & + \sum_{k_1<k_2-10}R_{\nu}^{(\alpha)}v_{k_1} R^{\nu,(\alpha)}v_{k_2}\\
 & + \sum_{k_2<k_1-10}R_{\nu}^{(\alpha)}v_{k_1} R^{\nu,(\alpha)}v_{k_2}. 
 \end{align*}
 It suffices to deal with the first and second term on the right hand side. This being quite standard in light of \cite{Tat1}, \cite{Tao1}, \cite{KrSch} for example, we only deal with the first term here. 
 \\
 
 {\it{(3.a): high-high interactions}}. Write this as 
 \begin{align*}
 R_{\nu}^{(\alpha)}v_{k_1} R^{\nu,(\alpha)}v_{k_2}& = \sum_{k<k_1+10}P_k \big[R_{\nu}^{(\alpha)}Q_{\geq k+10}v_{k_1} R^{\nu,(\alpha)}v_{k_2}\big]\\
 & + \sum_{k<k_1+10}P_k \big[R_{\nu}^{(\alpha)}Q_{<k+10}v_{k_1} R^{\nu,(\alpha)}Q_{\geq k+10}v_{k_2}\big]\\
 & + \sum_{k<k_1+10}P_k Q_{\geq k+10} \big[R_{\nu}^{(\alpha)}Q_{<k}v_{k_1} R^{\nu,(\alpha)}Q_{<k}v_{k_2}\big]\\
 & + \sum_{k<k_1+10}\sum_{j<k+10}P_k Q_j\big[R_{\nu}^{(\alpha)}Q_{<j}v_{k_1} R^{\nu,(\alpha)}Q_{<j}v_{k_2}\big]\\
 & +  \sum_{k<k_1+10}\sum_{j<k+10}P_k Q_{<j}\big[R_{\nu}^{(\alpha)}Q_{j}v_{k_1} R^{\nu,(\alpha)}Q_{<j}v_{k_2}\big]\\
 & +  \sum_{k<k_1+10}\sum_{j<k+10}P_k Q_{<j}\big[R_{\nu}^{(\alpha)}Q_{<j}v_{k_1} R^{\nu,(\alpha)}Q_{j}v_{k_2}\big]\\
\end{align*}
 Here the first and second terms as well as the fifth and sixth terms are essentially the same, of course. We shall here exploit the full generality of the spaces $N_k$ to estimate these terms. 
 \begin{itemize}
 \item {\it{The first term on the right.}} Note that for this term either the second factor is at modulation $>2^{k+5}$ or else the entire expression is at modulation $>2^{k+5}$. Thus we reduce to estimating 
 \begin{align*}
 &\big\|P_k \big[R_{\nu}^{(\alpha)}Q_{\geq k+10}v_{k_1} R^{\nu,(\alpha)}Q_{\geq k+5}v_{k_2}\big]\big\|_{L_t^1 \dot{H}^{s_{\alpha} - 1}}\\&\lesssim 
 2^{(\frac12 - 2\alpha)k}\cdot 2^{\frac32 k}2^{(1-\alpha)k_1}\big(2^{-k_1}\big\|\nabla_{t,x}Q_{\geq k+10}v_{k_1}\big\|_{L_{t,x}^2}\big)\cdot 2^{(1-\alpha)k_2}\big(2^{-k_2}\big\|\nabla_{t,x}Q_{\geq k+5}v_{k_2}\big\|_{L_{t,x}^2}\big)\\
 &\lesssim 2^{(\frac12 - 2\alpha)k}\cdot 2^{\frac32 k}\cdot 2^{(1-\alpha)k_1}\cdot 2^{-\frac{k}{2}}\cdot 2^{(2\alpha - \frac32)k_1}\cdot 2^{(1-\alpha)k_2}\cdot 2^{-\frac{k}{2}}\cdot 2^{(2\alpha - \frac32)k_2}\cdot\epsilon_*^2\\
 &\lesssim 2^{(1-2\alpha)(k-k_1)}\cdot\epsilon_*^2.\\
 \end{align*}
 This can be summed over $k<k_1+10$ provided $\alpha<\frac14$. In case that the whole expression is at modulation $\geq 2^{k+5}$, we place it into $\dot{X}^{s_{\alpha} - 1,-\frac12,1}$. 
 \item {\it{The third term on the right}}. Here we use null-frame spaces. We have 
 \begin{align*}
&P_k Q_{\geq k+10} \big[R_{\nu}^{(\alpha)}Q_{<k}v_{k_1} R^{\nu,(\alpha)}Q_{<k}v_{k_2}\big]\\
& = \sum_{\pm}\sum_{\kappa_1\sim -\kappa_2\in K_{k-k_1}}P_k Q_{k_1+O(1)} \big[R_{\nu}^{(\alpha)}Q_{<k}^{\pm}P_{k_1,\kappa_1}v R^{\nu,(\alpha)}Q_{<k}^{\pm}P_{k_2,\kappa_2}v\big],\\
 \end{align*}
 and so we can bound this by 
 \begin{align*}
 &\big\|P_k Q_{\geq k+10} \big[R_{\nu}^{(\alpha)}Q_{<k}v_{k_1} R^{\nu,(\alpha)}Q_{<k}v_{k_2}\big]\big\|_{\dot{X}^{s_\alpha - 1,-\frac12,1}}\\
 &\lesssim 2^{-\frac{k}{2}}\cdot 2^{(\frac12 - 2\alpha)k}\cdot 2^{(1-\alpha)k_1}\cdot\big(\sum_{\kappa\in K_{k-k_1}}\big\|P_{k_1,\kappa}Q_{<k}^{\pm}v\big\|_{PW[\pm\kappa]}^2\big)^{\frac12}\\
 &\hspace{4cm}\cdot 2^{(1-\alpha)k_2}\big(\sum_{\kappa\in K_{k-k_1}}\big\|P_{k_2,\kappa}Q_{<k}^{\pm}v\big\|_{NFA^*[\mp\kappa]}^2\big)^{\frac12}\\
 &\lesssim 2^{-\frac{k}{2}}\cdot 2^{(\frac12 - 2\alpha)k}\cdot 2^{(1-\alpha)k_1}\cdot 2^{(2\alpha - \frac12)k_1}\cdot 2^{k-k_1}\cdot 2^{(1-\alpha)k_2}\cdot 2^{(2\alpha - \frac32)k_2}\cdot \epsilon_*^2\\
 &\lesssim 2^{(1-2\alpha)(k-k_1)}\cdot \epsilon_*^2,\\
 \end{align*}
 which can be summed over $k<k_1 = k_2+O(1)$, provided $\alpha<\frac14$. 
 \item {\it{The fourth to sixth terms.}} Write 
 \begin{align*}
 2P_k Q_j\big[R_{\nu}^{(\alpha)}Q_{<j}v_{k_1} R^{\nu,(\alpha)}Q_{<j}v_{k_2}\big]& = \Box P_k Q_j\big[|\nabla|^{-\alpha}Q_{<j}v_{k_1} |\nabla|^{-\alpha}Q_{<j}v_{k_2}\big]\\
 & - P_k Q_j\big[|\nabla|^{-\alpha}Q_{<j}\Box v_{k_1} |\nabla|^{-\alpha}Q_{<j}v_{k_2}\big]\\
 & - P_k Q_j\big[|\nabla|^{-\alpha}Q_{<j}v_{k_1} |\nabla|^{-\alpha}Q_{<j}\Box v_{k_2}\big]\\
 \end{align*}
 The last two terms on the right are of course symmetrical, and it suffices to bound one of them. The first term on the right can be estimated purely by means of Strichartz estimates
 \begin{align*}
&\big\|\Box P_k Q_j\big[|\nabla|^{-\alpha}Q_{<j}v_{k_1} |\nabla|^{-\alpha}Q_{<j}v_{k_2}\big]\big\|_{\dot{X}_k^{s_\alpha - 1,-\frac12, 1}}\\
&\lesssim 2^{(\frac32 - 2\alpha)k}\cdot 2^{\frac{j}{2}}\cdot 2^{-\alpha k_1}\big\|Q_{<j}v_{k_1}\big\|_{L_t^3L_x^6}\cdot 2^{-\alpha k_2}\big\|Q_{<j}v_{k_2}\big\|_{L_t^6 L_x^3}\\
&\lesssim 2^{\frac{j-k}{2}}\cdot 2^{(2 - 2\alpha)k}\cdot 2^{-\alpha k_1}\cdot 2^{(2\alpha - \frac32 + \frac23)k_1}\cdot 2^{-\alpha k_2}\cdot 2^{(2\alpha - \frac32 + \frac13)k_2}\cdot\epsilon_*^2\\
& = 2^{\frac{j-k}{2}}\cdot 2^{(2 - 2\alpha)(k-k_1)}\cdot\epsilon_*^2.\\
 \end{align*}
 This can be summed over $j<k, k<k_1 = k_2+O(1)$(where we assume as usual $\alpha<\frac14$). 
 \\
 For the remaining terms in the above decomposition, we have to again resort to null-frame spaces: write 
 \begin{align*}
 &P_k Q_j\big[|\nabla|^{-\alpha}Q_{<j}\Box v_{k_1} |\nabla|^{-\alpha}Q_{<j}v_{k_2}\big]\\
 & = \sum_{\pm,\pm}\sum_{\substack{\kappa_{1,2}\in K_{k-k_1 +\frac{j-k}{2}}\\ \text{dist}(\pm\kappa_1,\pm\kappa_2)\sim 2^{k-k_1 +\frac{j-k}{2}}}}P_k Q_j\big[|\nabla|^{-\alpha}Q_{<j}^{\pm}\Box P_{k_1,\kappa_1}v |\nabla|^{-\alpha}Q_{<j}^{\pm}P_{k_2,\kappa_2}v\big]\\
 & =  \sum_{\pm,\pm,\pm}\sum_{\substack{\kappa_{1,2}\in K_{k-k_1 +\frac{j-k}{2}}\\ \text{dist}(\pm\kappa_1,\pm\kappa_2)\sim 2^{k-k_1 +\frac{j-k}{2}}}}\sum_{\substack{\kappa\in C_{\frac{j-k}{2}}\\ \text{dist}(\pm\kappa,\pm\kappa_1)\sim 2^{\frac{j-k}{2}}}}P_{k,\kappa} Q_j^{\pm}\big[|\nabla|^{-\alpha}Q_{<j}^{\pm}\Box P_{k_1,\kappa_1}v \\&\hspace{9cm}\cdot|\nabla|^{-\alpha}Q_{<j}^{\pm}P_{k_2,\kappa_2}v\big].\\
 \end{align*}
 To simplify notation, denote the second sum counting from the left by $\Sigma^{(1)}$ and the third one $\Sigma^{(2)}$. 
 It follows that 
 \begin{align*}
 &\big\|P_k Q_j\big[|\nabla|^{-\alpha}Q_{<j}\Box v_{k_1} |\nabla|^{-\alpha}Q_{<j}v_{k_2}\big]\big\|_{N_k}\\
 &\lesssim 2^{(\frac12 - 2\alpha)k}\sum_{\pm,\pm,\pm}\Sigma^{(1)}\big(\Sigma^{(2)}\big\|P_{k,\kappa} Q_j^{\pm}\big[|\nabla|^{-\alpha}Q_{<j}^{\pm}\Box P_{k_1,\kappa_1}v \cdot|\nabla|^{-\alpha}Q_{<j}^{\pm}P_{k_2,\kappa_2}v\big]\big\|_{NFA[\pm\kappa]}^2\big)^{\frac12}\\
 &\lesssim  2^{(\frac12 - 2\alpha)k}\sum_{\pm,\pm}\Sigma^{(1)}2^{-\frac{j-k}{2}}\big\||\nabla|^{-\alpha}Q_{<j}^{\pm}\Box P_{k_1,\kappa_1}v\big\|_{L_{t,x}^2}\cdot \big\||\nabla|^{-\alpha}Q_{<j}^{\pm}P_{k_2,\kappa_2}v\big\|_{PW[\pm\kappa_2]}.
 \end{align*}
 Here we have exploited that for fixed $\kappa_{1,2}$ there are only $O(1)$ many choices for $\kappa$. Since for fixed $\kappa_1$ there are only $O(1)$ many choices for $\kappa_2$ (in $\Sigma^{(1)}$), we can apply the Cauchy-Schwarz inequality as well as Plancherel's theorem to bound the preceding by 
 \begin{align*}
& 2^{(\frac12 - 2\alpha)k}\cdot 2^{-\frac{j-k}{2}}\sum_{\pm,\pm}\big\||\nabla|^{-\alpha}Q_{<j}^{\pm}\Box v_{k_1}\big\|_{L_{t,x}^2}\cdot \big(\sum_{\kappa_2}\big\||\nabla|^{-\alpha}Q_{<j}^{\pm}P_{k_2,\kappa_2}v\big\|_{PW[\pm\kappa_2]}^2\big)^{\frac12}\\
 &\lesssim  2^{(\frac12 - 2\alpha)k}\cdot 2^{-\frac{j-k}{2}}\cdot 2^{(1-\alpha)k_1}\cdot \cdot 2^{\frac{j}{2}}\cdot 2^{(2\alpha - \frac32)k_1}\cdot 2^{-\alpha k_2}\cdot 2^{\frac{j-k}{2}}\cdot 2^{(2\alpha - \frac12)k_2}\cdot \epsilon_*^2\\
 &\lesssim 2^{\frac{j-k}{2}}\cdot 2^{(1-2\alpha)(k-k_1)}\cdot \epsilon_*^2. 
 \end{align*}
 This can be summed over $j<k, k<k_1 = k_2+O(1)$.

 \end{itemize}

  \end{proof}


\centerline{\scshape Nicolas Burq}
\medskip
{\footnotesize
 \centerline{D\'{e}partement de Math\'ematiques}
\centerline{Universit\'{e} Paris-Sud}
\centerline{Bat. 307, 91405 Orsay Cedex
FRANCE}
\centerline{\email{nicolas.burq@u-psud.fr}
} 

}

\medskip

\centerline{\scshape Joachim Krieger }
\medskip
{\footnotesize
 \centerline{B\^{a}timent des Math\'ematiques, EPFL}
\centerline{Station 8, 
CH-1015 Lausanne, 
  Switzerland}
  \centerline{\email{joachim.krieger@epfl.ch}}
}

\end{document}